\documentclass[12pt,oneside,reqno]{amsart}

\usepackage{amsmath,amssymb,amsthm,textcomp}
\usepackage{dsfont}
\usepackage{amsfonts,graphicx}
\usepackage[mathscr]{eucal}
\usepackage{color}
\usepackage{csquotes}
\usepackage{relsize}
\usepackage{hyperref}
\usepackage[backend=bibtex,%
firstinits=true,%
doi=false,%
isbn=true,%
url=false,%
maxnames=99]{biblatex}%

\AtEveryBibitem{\clearfield{issn}}
\AtEveryCitekey{\clearfield{issn}}
\numberwithin{equation}{section}
\DeclareNameAlias{sortname}{last-first}
\theoremstyle{definition}
\usepackage{mathtools}
\numberwithin{equation}{section}

\newcommand{\ncom}{\newcommand}

\ncom{\beq}{\begin{equation}}
	\ncom{\eeq}{\end{equation}}
\ncom{\bea}{\begin{eqnarray*}}
	\ncom{\eea}{\end{eqnarray*}}
\ncom{\beqa}{\begin{eqnarray}}
	\ncom{\eeqa}{\end{eqnarray}}
\ncom{\nno}{\nonumber}
\ncom{\non}{\nonumber}
\ncom{\ds}{\displaystyle}
\ncom{\half}{\frac{1}{2}}
\ncom{\mbx}{\makebox{.25cm}}
\ncom{\hs}{\mbox{\hspace{.25cm}}}
\ncom{\rar}{\rightarrow}
\ncom{\Rar}{\Rightarrow}
\ncom{\noin}{\noindent}
\ncom{\bc}{\begin{center}}
	\ncom{\ec}{\end{center}}
\ncom{\sz}{\scriptsize}
\ncom{\rf}{\ref}
\ncom{\s}{\sqrt{2}}
\ncom{\sgm}{\sigma}
\ncom{\Sgm}{\Sigma}
\ncom{\psgm}{\sigma^{\prime}}
\ncom{\dt}{\delta}
\ncom{\Dt}{\Delta}
\ncom{\lmd}{\theta}
\ncom{\Lmd}{\Lambda}
\ncom{\Th}{\Theta}
\ncom{\e}{\eta}
\ncom{\eps}{\epsilon}
\ncom{\pcc}{\stackrel{P}{>}}
\ncom{\lp}{\stackrel{L_{p}}{>}}
\ncom{\dist}{{\rm\,dist}}
\ncom{\sspan}{{\rm\,span}}
\ncom{\re}{{\rm Re\,}}
\ncom{\im}{{\rm Im\,}}
\ncom{\sgn}{{\rm sgn\,}}
\ncom{\ba}{\begin{array}}
	\ncom{\ea}{\end{array}}
\ncom{\hone}{\mbox{\hspace{1em}}}
\ncom{\htwo}{\mbox{\hspace{2em}}}
\ncom{\hthree}{\mbox{\hspace{3em}}}
\ncom{\hfour}{\mbox{\hspace{4em}}}
\ncom{\vone}{\vskip 2ex}
\ncom{\vtwo}{\vskip 4ex}
\ncom{\vonee}{\vskip 1.5ex}
\ncom{\vthree}{\vskip 6ex}
\ncom{\vfour}{\vspace*{8ex}}
\ncom{\norm}{\|\;\;\|}
\ncom{\integ}[4]{\int_{#1}^{#2}\,{#3}\,d{#4}}
\ncom{\vspan}[1]{{{\rm\,span}\{ #1 \}}}
\ncom{\dm}[1]{ {\displaystyle{#1} } }
\ncom{\ri}[1]{{#1} \index{#1}}

\newtheorem{theorem}{\bf Theorem}[section]
\newtheorem{remark}{\bf Remark}[section]

\newtheorem{lemma}{Lemma}[section]

\newtheoremstyle
{remarkstyle}
{}
{11pt}
{}
{}
{\bfseries}
{:}
{     }
{\thmname{#1} \thmnumber{#2} }

\theoremstyle{remarkstyle}

\def\eps{\varepsilon}

\begin{document}
\title{Tempered Erlang Queue with Multiple Arrivals}
\author[Manisha Dhillon]{Manisha Dhillon}
\address{Manisha Dhillon, Department of Mathematics, Indian Institute of Technology Bhilai, Durg 491002, India.}
\email{manishadh@iitbhilai.ac.in}
\author[Kuldeep Kumar Kataria]{Kuldeep Kumar Kataria}
\address{Kuldeep Kumar Kataria, Department of Mathematics, Indian Institute of Technology Bhilai, Durg 491002, India.}
\email{kuldeepk@iitbhilai.ac.in}
\date{\today}
\begin{abstract}
In this paper, we introduce and study a time-changed variant of the Erlang queue with multiple arrivals where the time-changing component used is the first hitting time of a tempered stable subordinator. The system of fractional difference-differential equations that governs its state probabilities is derived which is solved to obtain their explicit expressions. An equivalent representation in terms of phases and the mean queue length is obtained. For a particular case, the distribution of inter-arrival times, inter-phase times, sojourn times,  busy period and that of conditional waiting times are derived. 
\end{abstract}
	
\maketitle 
\section{Introduction}
The Erlang queue is a widely studied model in queueing theory. It is characterized by Poisson arrivals with rate $\lambda$ and a service mechanism that follows Erlang distribution with shape parameter $k$ and mean $1/\mu$. Each customer undergoes $k$ exponentially distributed service phases with parameter $k\mu$. The theory and application of queuing system models have become an integral part of many domains, for example, it is used in call centers to estimate the required number of service agents (see Spath and F{\"a}hnrich (2006)), to model a healthcare process (see Fomundam and Herrmann (2007)), in modeling of financial data (see Cahoy {\it et al.} (2015)) and in telecommunications (see Giambene (2014)), \textit{etc.}

Di Crescenzo \textit{et al.} (2003) studied the $M/M/1$ queue with effect of catastrophes. Luchak (1956, 1958) considered single-channel queues with Poisson arrivals and general service-time distributions where the solutions for time-dependent state probabilities of these models are derived. Giorno \textit{et al.} (2018) discussed a single-server queue with Poisson arrivals and a state-dependent service mechanism that leads to a logarithmic steady-state distribution.
 
Cahoy \textit{et al.} (2015) introduced a fractional generalization of the $M/M/1$ queue whereas Ascione \textit{et al.} (2018) studied a  fractional variant of the $M/M/1$ queue with catastrophes. Griffiths \textit{et al.} (2006) obtained the transient solution of Erlang queue and Ascione \textit{et al.} (2020) time-changed it by an independent inverse stable subordinator. Chen \textit{et al.} (2020) discussed Markovian bulk-arrival and bulk-service queues with state-dependent control. Recently, Pote and Kataria (2025) studied a queue in which the service system has Erlang distribution and the arrivals are modeled by a generalized counting process (GCP), a counting process with the  possibility of finitely many arrivals in an infinitesimal time interval. For more details on GCP and some recent works on it, we refer the reader to Di Crescenzo \textit{et al.} (2016), Kataria and Khandakar (2022), Dhillon and Kataria (2024), and the references therein.

In this paper, we introduce and study a time-changed variant of the Erlang queue with multiple arrivals. The time is changed via the first hitting time of tempered stable subordinator, and we refer to this time-changed process as the tempered Erlang queue with multiple arrivals. The inverse tempered stable subordinator captures heavy-tailed waiting phenomena while still ensuring finite mean waiting times, thus bridging the gap between classical Markovian models and purely fractional models. This provides a more realistic description of queueing systems where bursts of arrivals or delays occur in environments subject to memory and long-range dependence, but without infinite expected waiting times. This characteristic makes it a more suitable alternative to the case of inverse stable subordinator, which typically yields infinite expected values. Consequently, it establishes a bridge between the classical Markovian setting and fractional generalizations.
This study is particularly relevant in performance modeling of modern communication and service systems such as call centers, computer networks, and cloud services, where the traffic is bursty and varies frequently. The subordination by an inverse tempered stable subordinator allows us to model anomalous fluctuations in the arrival process more accurately, thereby improving system design, resource allocation, and prediction of congestion phenomena. The paper is organized as follows: 

In Section \ref{preliminaries}, we give some definitions and known results on the Mittag-Leffler function, tempered fractional derivative, Erlang queue with multiple arrivals, \textit{etc.}  In Section \ref{sec3}, we study the Erlang queue with multiple arrivals time-changed by an independent inverse tempered $\alpha$-stable subordinator. We derive the system of fractional differential equations that governs its state probabilities, and obtain the explicit expression for its transient state probabilities. We deduce the probability of no customer in the system at time $t\ge0$. Also, we give an equivalent characterization of the tempered Erlang queue with multiple arrivals in terms of the number of phases, that is, we study its queue length process. We obtain the closed form  expression for the mean queue length and derive its governing system of differential equations. Moreover, we obtain the distribution of busy period which is defined as the time duration which starts when a customer enters in an empty system and ends when the system becomes empty again. In Section \ref{se4}, by restricting the number of arrivals to one in an infinitesimal small time interval, we study the tempered Erlang queue, that is, the Erlang queue time-changed by an independent inverse tempered $\alpha$-stable subordinator. For this time-changed process, we present its various distributional properties derived as a particular case from Section \ref{sec3}. Additionally, we obtain the distribution of inter-arrival times, inter-phase times and sojourn times of the tempered Erlang queue. Later, the conditional waiting time distribution for tempered Erlang queue is derived.

\section{Preliminaries}\label{preliminaries}  
Here, we collect some definitions and known results on Mittag-Leffler function, tempered stable subordinator and its inverse, fractional derivatives and Erlang queue with multiple arrivals that will be used in the subsequent sections.
\subsection{Mittag-Leffler function}
The three-parameter Mittag-Leffler function is defined as (see Kilbas {\it et al.} (2006), Eq. (1.9.1))
\begin{equation}\label{mitag}
E_{\alpha,\,\beta}^{\gamma}(x)=\frac{1}{\Gamma(\gamma)}\sum_{j=0}^{\infty} \frac{x^{j}\Gamma(j+\gamma)}{j!\Gamma(\alpha j+\beta)},\ x\in\mathbb{R},
\end{equation}
where $\alpha>0$, $\beta>0$ and $\gamma>0$. For $\gamma=1$, it reduces to two-parameter Mittag-Leffler function. Further, for $\gamma=\beta=1$, we get the Mittag-Leffler function. 

The following Laplace transform will be used (see Kilbas {\it et al.} (2006), Eq. (1.9.13)):
\begin{equation}\label{mi}
\mathcal{L}\big(t^{\beta-1}E^{\gamma}_{\alpha,\,\beta}(xt^{\alpha});s\big)=\frac{s^{\alpha\gamma-\beta}}{(s^{\alpha}-x)^{\gamma}},\ x\in\mathbb{R},\, s>|x|^{1/\alpha},\,  t>0.
\end{equation}
\subsection{Tempered stable subordinator and its inverse} A tempered stable subordinator $\{D_{\theta,\alpha}(t)\}_{t\geq0}$ with tempering parameter $\theta>0$ and stability index $\alpha\in(0,1)$ is  a non-decreasing L\'evy process. It is characterized by the following Laplace transform (see Meerschaert \textit{et al.} (2013)): 
\begin{equation*}
\mathbb{E}(e^{-sD_{\theta,\alpha}(t)})=e^{-t\Psi_{\theta,\alpha}(s)},\ s>0,
\end{equation*} 
where $\Psi_{\theta,\alpha}(s)=(s+\theta)^\alpha-\theta^\alpha$. Its first passage time, that is,  $Y_{\theta,\alpha}(t)=\inf\{x>0:D_{\theta,\alpha}(x)>t\}$, $t\ge0$ is known as the inverse tempered stable subordinator. 
Its density function $f_{\theta,\alpha}(t,x)$ has the following Laplace transform:
\begin{equation}\label{lapdeninvtemp}
\int_{0}^{\infty}e^{-st}f_{\theta,\alpha}(t,x)\ \mathrm{d}t=\frac{\Psi_{\theta,\alpha}(s)}{s}e^{-x\Psi_{\theta,\alpha}(s)}, \, s>0.
\end{equation}
%

\subsection{Fractional derivatives}
Let $\theta>0$ and $\alpha\in(0,1)$. The Riemann-Liouville fractional derivative  of a function $f$ is defined as
\begin{equation*}
\mathcal{D}_t^{\alpha} f(t)\coloneqq\frac{1}{\Gamma(1-\alpha)}\frac{\mathrm{d}}{\mathrm{d}t}\int_{0}^{t}\frac{f(s)}{(t-s)^\alpha}\mathrm{d}s,
\end{equation*} 
and the Riemann-Liouville tempered fractional derivative is defined as follows (see Alrawashdeh \textit{et al.} (2017), Eq. (2.6)):
\begin{equation*}
\mathcal{D}_t^{\theta,\alpha}f(t)\coloneqq e^{-\theta t}\mathcal{D}_t^{\alpha}(e^{\theta t}f(t))-\theta^\alpha f(t).
\end{equation*}

The Caputo tempered fractional derivative is defined by (see Alrawashdeh \textit{et al.} (2017), Eq. (2.8))
\begin{equation}\label{caputotemp}
\partial_t^{\theta,\alpha}f(t)=\mathcal{D}_t^{\theta,\alpha}f(t)-\frac{f(0)}{\Gamma(1-\alpha)}\int_{t}^{\infty}e^{-\theta s}\alpha s^{-\alpha -1}\mathrm{d}s.
\end{equation}
Its Laplace transform is given by (see Alrawashdeh \textit{et al.} (2017), Eq. (2.9))
\begin{equation}\label{laplacetempcapu}
\mathcal{L}(\partial_t^{\theta,\alpha}f(t);s)=\Psi_{\theta,\alpha}(s)\tilde{f}(s)-\frac{\Psi_{\theta,\alpha}(s)}{s}f(0),\, s>0,
\end{equation}
where $\Psi_{\theta,\alpha}(s)=(s+\theta)^{\alpha}-\theta^\alpha$.

\subsection{Erlang queue with multiple arrivals}
An Erlang queue with multiple arrivals (see Pote and Kataria (2025)) is a queueing system in which the arrival of customers is modeled according to a GCP in which $m$ customers arrive with transition rate $\lambda_m$, $m=1,2,\dots,l$ and the service system has Erlang distribution with shape parameter $k$ and rate $\mu$. Thus, the service system consist of $k$ phases such that each phase has exponential service time with parameter $k\mu$.

Let $\{\mathcal{Q}(t)\}_{t\ge0}$ denote the Erlang queue with multiple arrivals, that is, $\mathcal{Q}(t)=(\mathcal{N}(t),\mathcal{S}(t))$, $t\geq 0,$ with state space $S_{0}=S \cup \{(0,0)\}$, 
where $S=\{(n,s) \in \mathbb{N} \times \mathbb{N}: n \geq 1, \, 1 \leq s \leq k \}$. Here, $\mathcal{N}(t)$ denotes the number of customers in the system at time $t\ge0$ and $\mathcal{S}(t)$ denotes the phase of a customer being served at time $t$.
The state transition diagram of Erlang queue with two arrivals is depicted in Figure \ref{trans}.

For $t \geq 0$, let us denote the transient state probabilities as follows:
\begin{align*}
p_{0}(t)&=\mathrm{Pr}(\mathcal{Q}(t)=(0,0)|\mathcal{Q}(0)=(0,0)),\\
p_{n,s}(t)&=\mathrm{Pr}(\mathcal{Q}(t)=(n,s)|\mathcal{Q}(0)=(0,0)), \, (n,s) \in S,
\end{align*}
where $p_{0}(t)$ is the probability that there are no customers in the system at time $t$, and $p_{n,s}(t)$ is the probability of $n$ customers in the system at time $t$ and the customer being served is at phase $s$. 
For $|u|\le 1$, the generating function of $\{\mathcal{Q}(t)\}_{t\geq0}$ is given by
\begin{equation*}
G(u,t)=p_{0}(t)+\sum_{n=1}^{\infty}\sum_{s=1}^{k}u^{k(n-1)+s}p_{n,s}(t),\, t\ge0,
\end{equation*}
which is the solution of the following Cauchy problem (see Pote and Kataria (2025), Eq. (6)):
\begin{equation}\label{cauchyErlangq}
u\frac{\partial }{\partial t}G(u,t)=\Big(k\mu (1-u)-\Lambda u\Big(1-\sum_{m=1}^{l}c_m u^{mk}\Big)\Big)G(u,t)-k\mu (1-u)p_0(t), 
\end{equation}
where $\Lambda=\lambda_1+\lambda_2+\dots+\lambda_l$ and $c_m=\lambda_m/\Lambda$ for all $m=1,2,\dots,l$.
\section{Tempered Erlang queue with multiple arrivals}\label{sec3}
 Let $\{\mathcal{Q}(t)\}_{t\ge0}$ be the Erlang queue with arrival rates $\lambda_1$, $\lambda_2$, $\dots$, $\lambda_l$ such that it is independent of the inverse tempered $\alpha$-stable subordinator $\{Y_{\theta,\alpha}(t)\}_{t\geq0}$, $0<\alpha<1$, $\theta>0$. We introduce a time-changed variant of $\{\mathcal{Q}(t)\}_{t\ge0}$ as follows: $
 	\mathcal{Q}^{\theta,\alpha}(t)\coloneqq\mathcal{Q}(Y_{\theta,\alpha}(t))
 $ with $\mathcal{Q}^{0,1}(t)=\mathcal{Q}(t)$. 
We call the time-changed process  $\{\mathcal{Q}^{\theta,\alpha}(t)\}_{t\ge0}$ as the tempered Erlang queue with multiple arrivals. 

For $t\geq0$, its state probabilities are denoted by
\begin{align*}
p_{0}^{\theta,\alpha}(t)&=\mathrm{Pr}(\mathcal{Q}^{\theta,\alpha}(t)=(0,0)|\mathcal{Q}^{\theta,\alpha}(0)=(0,0)),\\
p_{n,s}^{\theta,\alpha}(t)&=\mathrm{Pr}(\mathcal{Q}^{\theta,\alpha}(t)=(n,s)|\mathcal{Q}^{\theta,\alpha}(0)=(0,0)),\, (n,s)\in S,
\end{align*}
where $S=\{(n,s) \in \mathbb{N} \times \mathbb{N}: n \geq 1, \, 1 \leq s \leq k \}$.	

Let us set the following notations: $\Lambda=\lambda_1+\lambda_2+\dots+\lambda_l$, $\beta=\theta^\alpha-\Lambda-k\mu$ and $\psi_{\theta,\alpha}(z)=(z+\theta)^\alpha-\theta^\alpha$, $z>0$.
For $\textbf{m}=(m_1,m_2,\dots,m_l)$, $\textbf{m}'=(m'_{1},m'_{2},\dots,m'_{l})$, $\textbf{m}''=(m_{1}+1,m_{2},\dots,m_{l})$, where $m_i\in\mathbb{N}_0$, let us denote $a_{r}^{n,s}(\textbf{m})=\sum_{j=1}^{l}m_{j}+k(r+1)-s+1$, $\pi_{r}^{n,s}(\textbf{m})=\alpha(a_{r}^{n,s}(\textbf{m})-1)+1$, 
\begin{align*}
\gamma_{h,n}^{0}(\textbf{m})&=h+n+k\sum_{j=1}^{l}jm_{j},\\
C_{h,n}^{0}(\textbf{m})&=h\Lambda^{n}\Big(\prod_{i=1}^{l}\frac{c_{i}^{m_{i}}}{m_{i}!}\Big)\frac{(k\mu)^{\gamma_{h,n}^{0}(\textbf{m})-n-1}}{(\gamma_{h,n}^{0}(\textbf{m})-n)!}(\gamma_{h,n}^{0}(\textbf{m})-1)!,\\ \beta_{h,n}^{0}(\textbf{m})&=\alpha(\gamma_{h,n}^{0}(\textbf{m})-1)+1,\\
A_{r}^{n,s}(\textbf{m})&=\Big(\prod_{i=1}^{l}\frac{c_{i}^{m_{i}}}{m_{i}!}\Big)\frac{\Lambda^{\sum_{j=1}^{l}m_{j}}(k\mu )^{k(r+1)-s}}{(k(r+1)-s)!}(a_{r}^{n,s}(\textbf{m})-1)!,\\
B_{r,h,w}^{n,s}(\textbf{m},\textbf{m}')&=k\mu A_{r}^{n,s}(\textbf{m})C_{h,w}^{0}(\textbf{m}'),\\
C_{r,h,w}^{n,s}(\textbf{m},\textbf{m}')&=\begin{cases}
k\mu A_{r}^{n,s+1}(\textbf{m})C_{h,w}^{0}(\textbf{m}'),\,s=1,2,\dots,k-1,\\
k\mu A_{r}^{n,1}(\textbf{m}'')C_{h,w}^{0}(\textbf{m}'),\, s=k,
\end{cases}	\\
c_{r,h,w}^{n,s}(\textbf{m},\textbf{m}')&=\begin{cases}
a_{r}^{n,s+1}(\textbf{m})+\gamma_{h,w}^{0}(\textbf{m}'),\,s=1,2,\dots,k-1,\\
a_{r}^{n,1}(\textbf{m}'')+\gamma_{h,w}^{0}(\textbf{m}'),\, s=k,
\end{cases}
\end{align*}
 $b_{r,h,w}^{n,s}(\textbf{m},\textbf{m}')=a_{r}^{n,s}(\textbf{m})+\gamma_{h,w}^{0}(\textbf{m}')$, $\rho_{r,h,w}^{n,s}(\textbf{m},\textbf{m}')=\alpha(b_{r,h,w}^{n,s}(\textbf{m},\textbf{m}')-1)+1$
and 
$\delta_{r,h,w}^{n,s}=\alpha(c_{r,h,w}^{n,s}(\textbf{m},\textbf{m}')-1)+1$, where  $n$, $r$, $w$ are non-negative integers, $h$ is a positive integer and $c_i=\lambda_i/\Lambda$.

\begin{figure}[htp]
	\includegraphics[width=12cm]{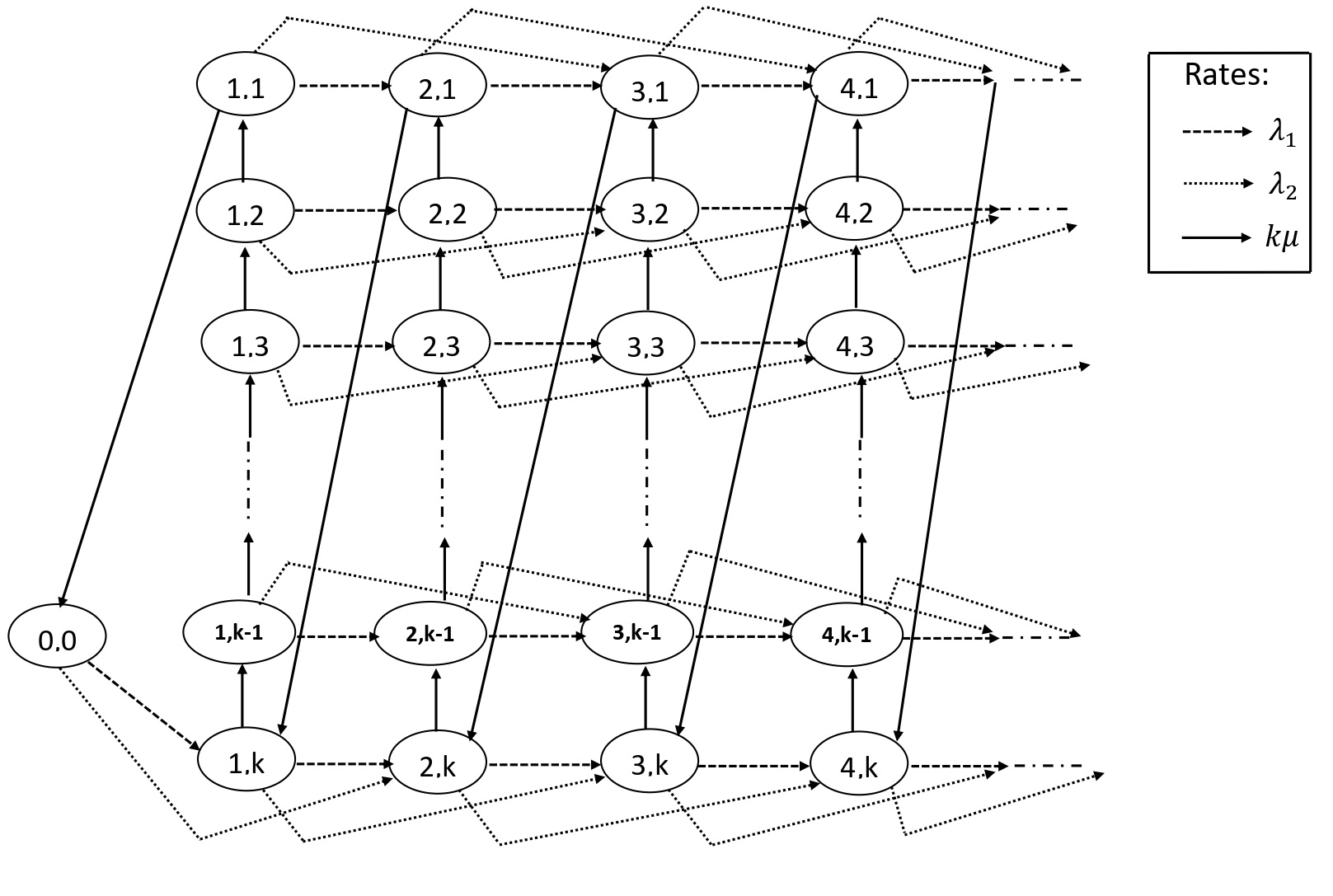}
	\caption{State transition diagram of $\{\mathcal{Q}(t)\}_{t\ge0}$ and $\{\mathcal{Q}^{\theta,\alpha}(t)\}_{t\ge0}$ for $l=2$.}\label{trans}
\end{figure}

In Figure \ref{trans}, the state transition diagram of tempered Erlang queue with two types of arrivals is depicted. From initial state $(0,0)$, transitions to states $(1,k)$ and $(2,k)$ occur with rates $\lambda_1$ and $\lambda_2$, respectively. For $n \geq 1$, transitions occur from $(n,s)$ to $(n+1,s)$ with rate $\lambda_1$ and from $(n,s)$ to $(n+2,s)$ with rate $\lambda_2$. Also, for $1 < s \leq k$, the process moves from $(n,s)$ to $(n,s-1)$ with rate $k\mu$. Moreover, for $n \geq 2$, a transition from $(n,1)$ to $(n-1,k)$ takes place with rate $k\mu$, and from state $(1,1)$ the system returns to initial state $(0,0)$ with  rate $k\mu$. 

First, we obtain the system of fractional differential equations that governs the state probabilities of tempered Erlang queue with multiple arrivals.
\begin{theorem}\label{cptfde}
The state probabilities of $\{\mathcal{Q}^{\theta,\alpha}(t)\}_{t\geq 0}$ solve the following system of fractional difference-differential equations:
\begin{align*}
\partial_t^{\theta,\alpha}p_{0}^{\theta,\alpha}(t)
&=-\Lambda p_{0}^{\theta,\alpha}(t)+k\mu p_{1,1}^{\theta,\alpha}(t),  \\
\partial_t^{\theta,\alpha}p_{1,s}^{\theta,\alpha}(t)
&=-(\Lambda + k \mu)p_{1,s}^{\theta,\alpha}(t)+k\mu p_{1,s+1}^{\theta,\alpha}(t), \, 1 \leq s \leq k-1, \\
\partial_t^{\theta,\alpha}p_{1,k}^{\theta,\alpha}(t)
&=-(\Lambda + k \mu)p_{1,k}^{\theta,\alpha}(t)+k\mu p_{2,1}^{\theta,\alpha}(t)+ \Lambda c_{1}p_{0}^{\theta,\alpha}(t),  \\
\partial_t^{\theta,\alpha}p_{n,s}^{\theta,\alpha}(t)
&=-(\Lambda + k \mu)p_{n,s}^{\theta,\alpha}(t)+k\mu p_{n,s+1}^{\theta,\alpha}(t)+\Lambda \sum_{m=1}^{\text{min}\{n,l\}}c_{m}p_{n-m,s}^{\theta,\alpha}(t), \,  1\leq s \leq k-1,  \, n \geq 2,\\
\partial_t^{\theta,\alpha}p_{n,k}^{\theta,\alpha}(t)
&=-(\Lambda + k \mu)p_{n,k}^{\theta,\alpha}(t)+k\mu p_{n+1,1}^{\theta,\alpha}(t)+\Lambda \sum_{m=1}^{n-1}c_{m}p_{n-m,k}^{\theta,\alpha}(t)+\Lambda c_{n}p_{0}^{\theta,\alpha}(t),\, 2\leq n\leq l, \\
\partial_t^{\theta,\alpha}p_{n,k}^{\theta,\alpha}(t)&=-(\Lambda + k \mu)p_{n,k}^{\theta,\alpha}(t)+k\mu p_{n+1,1}^{\theta,\alpha}(t)+ \Lambda \sum_{m=1}^{l}c_{m}p_{n-m,k}^{\theta,\alpha}(t),\,n >l,
\end{align*}
with $p_{0}^{\theta,\alpha}(0)=1$ and
$p_{n,s}^{\theta,\alpha}(0)=0$, $1\leq s\leq k$, $n\geq 1$, where $\partial_t^{\theta,\alpha}$ is the Caputo tempered fractional derivative as defined in \eqref{caputotemp}.
\end{theorem}
\begin{proof}		
For $|u|\le1$, let 
\begin{equation}\label{pgtemp}
G^{\theta,\alpha}(u,t)=p_{0}^{\theta,\alpha}(t)+\sum_{n=1}^{\infty} \sum_{s=1}^{k}u^{(n-1)k+s}p_{n,s}^{\theta,\alpha}(t), \ t\geq0,
\end{equation}
be the pgf of $\{\mathcal{Q}^{\theta,\alpha}(t)\}_{t\geq 0}$.
Also, for $(n,s) \in S$, we have
\begin{align}\label{pnsalphat}
p_{n,s}^{\theta,\alpha}(t)
&=\mathrm{Pr}(\mathcal{Q}(Y_{\theta,\alpha}(t))=(n,s)|\mathcal{Q}(0)=(0,0))\nonumber\\
&=\int_{0}^{\infty}\mathrm{Pr}(\mathcal{Q}(y)=(n,s)|\mathcal{Q}(0)=(0,0))\mathrm{Pr}(Y_{\theta,\alpha}(t)\in\mathrm{d}y)\nonumber\\
&=\int_{0}^{\infty}p_{n,s}(y)\mathrm{Pr}(Y_{\theta,\alpha}(t)\in\mathrm{d}y)
\end{align}
and
\begin{equation}\label{p0mu}
{p}_{0}^{\theta,\alpha}(t)=\int_{0}^{\infty}p_{0}(y)\mathrm{Pr}(Y_{\theta,\alpha}(t)\in\mathrm{d}y),
\end{equation}
where $p_{n,s}(y)$ and $p_{0}(y)$ are the transient \textit{state-phase} probability and zero \textit{state-phase} probability of Erlang queue with multiple arrivals, respectively.

On substituting \eqref{pnsalphat} and \eqref{p0mu} in \eqref{pgtemp}, we get
\begin{equation}\label{Gxtmu}
{G}^{\theta,\alpha}(u,t)=\int_{0}^{\infty}G(u,y)\mathrm{Pr}(Y_{\theta,\alpha}(t)\in\mathrm{d}y).
\end{equation}
By using \eqref{lapdeninvtemp}, the Laplace transforms of \eqref{p0mu} and \eqref{Gxtmu} are given by
\begin{equation}\label{p0nul}
\Tilde{p}_{0}^{\theta,\alpha}(z)=\frac{\psi_{\theta,\alpha}(z)}{z}\int_{0}^{\infty}p_{0}(y)e^{-y\psi_{\theta,\alpha}(z)}\mathrm{d}y
\end{equation}
and
\begin{equation}\label{Gxtnul}
\Tilde{G}^{\theta,\alpha}(u,z)=\frac{\psi_{\theta,\alpha}(z)}{z}\int_{0}^{\infty}G(u,y)e^{-y\psi_{\theta,\alpha}(z)}\mathrm{d}y,
\end{equation}
respectively.
		
By multiplying the first equation in the system of difference-differential equations given in Theorem \ref{cptfde} by $u$, the second equation by $u^{s+1}$, the third equation by $u^{k+1}$, the fourth by $u^{k(n-1)+s+1}$, the fifth and sixth by $u^{nk+1}$, and then summing all these equations over the range of $s=1,2,\dots,k$ and $n\ge1$, it can be established that the state probabilities of $\{\mathcal{Q}^{\theta,\alpha}(t)\}_{t\geq0}$ satisfy the system of equations in Theorem \ref{cptfde} if and only if its probability generating function (pgf) solves
\begin{equation}\label{cachyeq}
u\partial_t^{\theta,\alpha}G^{\theta,\alpha}(u,t)=\Big(k\mu(1-u)-\Lambda u\Big(1-\sum_{m=1}^{l}c_{m}u^{mk}\Big)\Big)G^{\theta,\alpha}(u,t)-k\mu(1-u)p_{0}^{\theta,\alpha}(t),
\end{equation}
with $G^{\theta,\alpha}(u,0)=1$.
		  
By taking the Laplace transform on both sides of \eqref{cachyeq} and using \eqref{laplacetempcapu}, we obtain
\begin{equation}\label{cauchyLL}
\Big(u\psi_{\theta,\alpha}(z)-k\mu(1-u)+\Lambda u\Big(1-\sum_{m=1}^{l}c_{m}u^{mk}\Big)\Big)\tilde{G}^{\theta,\alpha}(u,z)=\frac{u}{z}\psi_{\theta,\alpha}(z)-k\mu(1-u)\tilde{p}_0^{\theta,\alpha}(z).
\end{equation}
Now, by using \eqref{p0nul} and \eqref{Gxtnul} in \eqref{cauchyLL}, we obtain {\footnotesize\begin{equation}\label{lptcdee}
\int_{0}^{\infty}\Big(\Big(u\psi_{\theta,\alpha}(z)-k\mu(1-u)+\Lambda u\Big(1-\sum_{m=1}^{l}c_{m}u^{mk}\Big)\Big)G(u,y)+k\mu(1-u)p_{0}(y)\Big)\frac{\psi_{\theta,\alpha}(z)}{z}e^{-y\psi_{\theta,\alpha}(z)}\mathrm{d}y=\frac{u}{z}\psi_{\theta,\alpha}(z).
\end{equation}}
On using \eqref{cauchyErlangq} in \eqref{lptcdee}, we obtain the following identity: 
\begin{equation*}
\int_{0}^{\infty}\Big(u\psi_{\theta,\alpha}(z)G(u,y)-u \frac{\partial }{\partial y}G(u,y)\Big)\frac{\psi_{\theta,\alpha}(z)}{z}e^{-y\psi_{\theta,\alpha}(z)}\mathrm{d}y=\frac{u}{z}\psi_{\theta,\alpha}(z).
\end{equation*}
Thus, the required result holds true.
\end{proof}
\begin{remark}
On substituting $l=1$ in Theorem \ref{cptfde}, we obtain governing system of difference-differential equations of the state probabilities of Erlang queue time-changed with an independent inverse tempered $\alpha$-stable subordinator. Its explicit expression is given in \eqref{deertemp}. 
\end{remark} 
Next, we obtain the probability that there are no customers in the system at time $t$. 
\begin{theorem}\label{theorem4.2}
The zero state probability $p_{0}^{\theta,\alpha}(t)$ of $\{\mathcal{Q}^{\theta,\alpha}(t)\}_{t\ge0}$ is given by
\begin{align*}
p_{0}^{\theta,\alpha}(t)&=\sum_{h=1}^{\infty}\sum_{n=0}^{\infty}\sum_{\substack{\sum_{j=1}^{l}m_{j}=n\\m_{j}\geq0}}\sum_{m=0}^{\infty}C_{h,n}^{0}(\textbf{m})e^{-\theta t}\theta^mt^{\beta_{h,n}^{0}(\textbf{m})+m-1}\\
&\hspace{4cm} \cdot\Big(E_{\alpha,\beta_{h,n}^{0}(\textbf{m})+m}^{\gamma_{h,n}^{0}(\textbf{m})}(\beta t^{\alpha})-(\theta t)^\alpha E_{\alpha,\beta_{h,n}^{0}(\textbf{m})+\alpha+m}^{\gamma_{h,n}^{0}(\textbf{m})}(\beta t^{\alpha})\Big),
\end{align*}
where $E_{\alpha,\beta}^{\gamma}(\cdot)$ is the three parameter Mittag-Leffler function defined in \eqref{mitag}.
\end{theorem}
\begin{proof}
From \eqref{p0nul}, we have 
{\small\begin{align}
\Tilde{p}_{0}^{\theta,\alpha}(z)
&=\sum_{h=1}^{\infty}\sum_{n=0}^{\infty}\sum_{\substack{\sum_{j=1}^{l}m_{j}=n\\m_{j}\geq0}}h\Lambda^{n}\Big(\prod_{i=1}^{l}\frac{c_{i}^{m_{i}}}{m_{i}!}\Big)\frac{(k\mu)^{h-1+k\sum_{j=1}^{l}jm_{j}}}{(h+k\sum_{j=1}^{l}jm_{j})!}\frac{\psi_{\theta,\alpha}(z)}{z}\nonumber\\
&\hspace{6cm}\cdot\int_{0}^{\infty}y^{h+n+k\sum_{j=1}^{l}jm_{j}-1}e^{-(\Lambda+k\mu+\Psi_{\theta,\alpha}(z))y}\, \mathrm{d}y\nonumber\\
&=\sum_{h=1}^{\infty}\sum_{n=0}^{\infty}\sum_{\substack{\sum_{j=1}^{l}m_{j}=n\\m_{j}\geq0}}C_{h,n}^{0}(\textbf{m})\frac{z^{-1}\psi_{\theta,\alpha}(z)}{(\Lambda+k\mu+\psi_{\theta,\alpha}(z))^{h+n+k\sum_{j=1}^{l}jm_{j}}}\nonumber\\
&=\sum_{h=1}^{\infty}\sum_{n=0}^{\infty}\sum_{\substack{\sum_{j=1}^{l}m_{j}=n\\m_{j}\geq0}}\frac{C_{h,n}^{0}(\textbf{m})}{z}\bigg(\frac{(z+\theta)^\alpha}{((z+\theta)^\alpha-\beta)^{\gamma^0_{h,n}(\textbf{m})}}-\frac{\theta^\alpha}{((z+\theta)^\alpha-\beta)^{\gamma^0_{h,n}(\textbf{m})}}\bigg)\label{p0alpha4},
\end{align}	}
where the first equality follows on using Eq. (15) of Pote and Kataria (2025). Let $F(z)=(z+\theta)^\alpha((z+\theta)^\alpha-\beta)^{-\gamma^0_{h,n}(\textbf{m})}$
and $G(z)=\theta^\alpha((z+\theta)^\alpha-\beta)^{-\gamma^0_{h,n}(\textbf{m})}$. Then, the inverse Laplace transform of $F(z)$ and $G(z)$ follows from \eqref{mi}, and the shifting property of Laplace transform, that is, $F'(z)=F(z+\theta)$ and $G'(z)=G(z+\theta)$ leads to the inverse Laplace transform of $F'(z)$ and $G'(z)$. It is given by
\begin{align*}
\mathcal{L}^{-1}(F'(z))&=e^{-\theta t}\mathcal{L}^{-1}(F(z))\\
&=e^{-\theta t}t^{\alpha(\gamma^0_{h,n}(\textbf{m})-1)-1}E_{\alpha,\alpha(\gamma^0_{h,n}(\textbf{m})-1)}^{\gamma^0_{h,n}(\textbf{m})}(\beta t^\alpha)\\
&=e^{-\theta t}t^{B^0_{h,n}(\textbf{m})-2}E_{\alpha,B^0_{h,n}(\textbf{m})-1}^{\gamma^0_{h,n}(\textbf{m})}(\beta t^\alpha)
\end{align*}
and 
\begin{align*}
\mathcal{L}^{-1}(G'(z))&=e^{-\theta t}\mathcal{L}^{-1}(G(z))\\
&=\theta^\alpha e^{-\theta t}t^{\alpha\gamma^0_{h,n}(\textbf{m})-1}E_{\alpha,\alpha\gamma^0_{h,n}(\textbf{m})}^{\gamma^0_{h,n}(\textbf{m})}(\beta t^\alpha)\\
&=\theta^\alpha e^{-\theta t}t^{B^0_{h,n}(\textbf{m})+\alpha-2}E_{\alpha,B^0_{h,n}(\textbf{m})+\alpha-1}^{\gamma^0_{h,n}(\textbf{m})}(\beta t^\alpha),
\end{align*}
respectively. Further,
\begin{align}
\mathcal{L}^{-1}\Big(\frac{F'(z)}{z}\Big)&=\int_{0}^{t}e^{-\theta y}y^{B^0_{h,n}(\textbf{m})-2}E_{\alpha,B^0_{h,n}(\textbf{m})-1}^{\gamma^0_{h,n}(\textbf{m})}(\beta y^\alpha)\ \mathrm{d}y\nonumber\\
&=e^{-\theta t}\sum_{m=0}^{\infty}\frac{\theta^m}{m!}\int_{0}^{t}(t-y)^my^{B^0_{h,n}(\textbf{m})-2}E_{\alpha,B^0_{h,n}(\textbf{m})-1}^{\gamma^0_{h,n}(\textbf{m})}(\beta y^\alpha)\ \mathrm{d}y\nonumber\\
&=e^{-\theta t}\sum_{m=0}^{\infty}\theta^m t^{B^0_{h,n}(\textbf{m})+m-1}E_{\alpha,B^0_{h,n}(\textbf{m})+m}^{\gamma^0_{h,n}(\textbf{m})}(\beta t^\alpha),\label{inlshi1}
\end{align}
where in the last step, we have used the following result (see Kilbas \textit{et al.} (2004)): 
\begin{equation*}
\int_{0}^{t}y^{\mu -1}E_{\rho,\mu}^{\beta}(wy^\rho)(t-y)^{\nu-1}\ \mathrm{d}y=\Gamma(\nu)t^{\nu+\mu-1}E_{\rho,\mu+\nu}^{\beta}(wt^\rho).
\end{equation*}
Similarly, we have
\begin{equation}\label{inlshi2}
\mathcal{L}^{-1}\Big(\frac{G'(z)}{z}\Big)
=\theta^\alpha e^{-\theta t}\sum_{m=0}^{\infty}\theta^m t^{B^0_{h,n}(\textbf{m})+\alpha+m-1}E_{\alpha,B^0_{h,n}(\textbf{m})+\alpha+m}^{\gamma^0_{h,n}(\textbf{m})}(\beta t^\alpha).
\end{equation}
Finally, on taking the inverse Laplace transform of \eqref{p0alpha4}, and using \eqref{inlshi1} and \eqref{inlshi2}, we get the required result.
\end{proof}
\begin{remark}
On taking $l=1$ in Theorem \ref{theorem4.2}, we get the corresponding result for Erlang queue time-changed by an independent inverse tempered stable subordinator. Its explicit expression is given in \eqref{ertempzero}.
\end{remark}

\begin{lemma}\label{lemma4.1}
Let $p_{0}(t)$ be the zero {\it state-phase} probability of Erlang queue with multiple arrivals. Then, 
{\footnotesize\begin{equation*}
\int_{0}^{\infty}\int_{0}^{s}p_{0}(t)(s-t)^{r}e^{-(\Lambda+k\mu)(s-t)}e^{-s\psi_{\theta,\alpha}(z)}\,\mathrm{d}t\,\mathrm{d}s=\sum_{h=1}^{\infty}\sum_{w=0}^{\infty}\sum_{\substack{\sum_{j=1}^{l}m_{j}'=w\\m_{j}'\geq0}}\frac{C_{h,w}^{0}(\textbf{m}')\Gamma(r+1)}{(\Lambda+k\mu+\psi_{\theta,\alpha}(z))^{\gamma_{h,w}^{0}(\textbf{m}')+r+1}},
\end{equation*}}
where $r\in\mathbb{N}_0$.
\end{lemma}
\begin{proof}
On using \eqref{p0nul}, the proof follows similar lines to that of Lemma 2 of Pote and Kataria (2025). Hence, it is omitted.
\end{proof}

\begin{theorem}\label{thm4.3}
For $n\geq1$ and $1\leq s\leq k$, the state probabilities of $\{\mathcal{Q}^{\theta, \alpha}(t)\}_{t\geq0}$ are given by  		
{\footnotesize \begin{align*}
p_{n,s}^{\theta,\alpha}(t)&=\sum_{\substack{m_{j},r\geq0\\\sum_{j=1}^{l}jm_{j} -r=n}}A_{r}^{n,s}(\textbf{m})\sum_{a=0}^{\infty}e^{-\theta t}\theta^a t^{\pi_{r}^{n,s}(\textbf{m})+a-1}\Big(E_{\alpha,\pi_{r}^{n,s}(\textbf{m})+a}^{a_{r}^{n,s}(\textbf{m})}(\beta t^{\alpha})-(\theta t)^\alpha E_{\alpha,\pi_{r}^{n,s}(\textbf{m})+\alpha +a}^{a_{r}^{n,s}(\textbf{m})}(\beta t^{\alpha})\Big)\\
 &\ \ +\sum_{\substack{m_{j},r\geq0\\\sum_{j=1}^{l}jm_{j} -r=n}}\sum_{h=1}^{\infty}\sum_{w=0}^{\infty}\sum_{\substack{\sum_{j=1}^{l}m'_{j}=w\\m'_{j}\geq0}}B_{r,h,w}^{n,s}(\textbf{m},\textbf{m}')\sum_{b=0}^{\infty}e^{-\theta t}\theta^bt^{\rho_{r,h,w}^{n,s}(\textbf{m},\textbf{m}')+b-1}\Big(E_{\alpha,\rho_{r,h,w}^{n,s}(\textbf{m},\textbf{m}')+b}^{b_{r,h,w}^{n,s}(\textbf{m},\textbf{m}')}(\beta t^{\alpha})\\
 &\ \ -(\theta t)^\alpha E_{\alpha,\rho_{r,h,w}^{n,s}(\textbf{m},\textbf{m}')+\alpha +b}^{b_{r,h,w}^{n,s}(\textbf{m},\textbf{m}')}(\beta t^{\alpha})\Big) -\sum_{\substack{m_{j},r\geq0\\ \sum_{j=1}^{l}jm_{j} -r=n}}\sum_{h=1}^{\infty}\sum_{w=0}^{\infty}\sum_{\substack{\sum_{j=1}^{l}m'_{j}=w\\m'_{j}\geq0}}C_{r,h,w}^{n,s}(\textbf{m},\textbf{m}')\sum_{c=0}^{\infty}e^{-\theta t}\theta^c\\
 &\hspace{2cm} \cdot t^{\delta_{r,h,w}^{n,s}(\textbf{m},\textbf{m}')+c-1}\Big(E_{\alpha,\delta_{r,h,w}^{n,s}(\textbf{m},\textbf{m}')+c}^{c_{r,h,w}^{n,s}(\textbf{m},\textbf{m}')}(\beta t^{\alpha})-(\theta t)^\alpha E_{\alpha,\delta_{r,h,w}^{n,s}(\textbf{m},\textbf{m}')+\alpha +c}^{c_{r,h,w}^{n,s}(\textbf{m},\textbf{m}')}(\beta t^{\alpha}) \Big),
 \end{align*}}
where $E_{\alpha,\beta}^{\gamma}(\cdot)$ is the three parameter Mittag-Leffler function defined in \eqref{mitag}.
\end{theorem}
\begin{proof}
On taking the Laplace transform of \eqref{pnsalphat}, we get 
\begin{equation}\label{lppnst}
\Tilde{p}_{n,s}^{\theta,\alpha}(z)=\frac{\Psi_{\theta,\alpha}(z)}{z}\int_{0}^{\infty}p_{n,s}(y)e^{-y\Psi_{\theta,\alpha}(z)}\mathrm{d}y,
\end{equation}
where we have used \eqref{lapdeninvtemp}.
For $1\leq s\leq k-1$, by using Eq. (17) of Pote and Kataria (2025), we have
{\scriptsize	
\begin{align}
\Tilde{p}_{n,s}^{\theta,\alpha}(z)
&=\sum_{\substack{m_{j},r\geq0\\\sum_{j=1}^{l}jm_{j} -r=n}}\frac{\Psi_{\theta,\alpha}(z)}{z}\Big(\prod_{i=1}^{l}\frac{c_{i}^{m_{i}}}{m_{i}!}\Big)\frac{\Lambda^{\sum_{j=1}^{l}m_{j}}(k\mu )^{k(r+1)-s}}{(k(r+1)-s)!} \Big(\int_{0}^{\infty}y^{\sum_{j=1}^{l}m_{j}+k(r+1)-s}e^{-(\Lambda+k\mu+\Psi_{\theta,\alpha}(z))y}\mathrm{d}y\nonumber\\
&\ \  +k\mu \int_{0}^{\infty}\int_{0}^{y}p_{0}(v)(y-v)^{\sum_{j=1}^{l}m_{j}+k(r+1)-s} e^{-(\Lambda+k\mu)(y-v)} e^{-y\Psi_{\theta,\alpha}(z)}\,\mathrm{d}v\mathrm{d}y\Big) - \sum_{\substack{m_{j},r\geq0\\\sum_{j=1}^{l}jm_{j} -r=n}}\frac{\Psi_{\theta,\alpha}(z)}{z}\nonumber\\
&\ \ \cdot\Big(\prod_{i=1}^{l}\frac{c_{i}^{m_{i}}}{m_{i}!}\Big)\frac{\Lambda^{\sum_{j=1}^{l}m_{j}}(k\mu )^{k(r+1)-s}}{(k(r+1)-s-1)!} \int_{0}^{\infty}\int_{0}^{y}p_{0}(v)(y-v)^{\sum_{j=1}^{l}m_{j}+k(r+1)-s-1}e^{-(\Lambda+k\mu)(y-v)}e^{-y\Psi_{\theta,\alpha}(z)}\,\mathrm{d}v\mathrm{d}y\,\,\nonumber\\
&=\frac{(z+\theta)^\alpha-\theta^\alpha}{z}\Bigg(\sum_{\substack{m_{j},r\geq0\\\sum_{j=1}^{l}jm_{j} -r=n}}\frac{A_{r}^{n,s}(\textbf{m})}{((z+\theta)^\alpha-\beta)^{a_r^{n,s}(\textbf{m})}}+\sum_{\substack{m_{j},r\geq0\\\sum_{j=1}^{l}jm_{j} -r=n}}\sum_{h=1}^{\infty}\sum_{w=0}^{\infty}\sum_{\substack{\sum_{j=1}^{l}m'_{j}=w\\m'_{j}\geq0}}  \nonumber\\
&\hspace{1.1cm} \cdot \frac{B_{r,h,w}^{n,s}(\textbf{m},\textbf{m}')}{((z+\theta)^\alpha-\beta)^{b_{r,h,w}^{n,s}(\textbf{m},\textbf{m}')}}-\sum_{\substack{m_{j},r\geq0\\\sum_{j=1}^{l}jm_{j} -r=n}}\sum_{h=1}^{\infty}\sum_{w=0}^{\infty}\sum_{\substack{\sum_{j=1}^{l}m'_{j}=w\\m'_{j}\geq0}}\frac{C_{r,h,w}^{n,s}(\textbf{m},\textbf{m}')}{((z+\theta)^\alpha-\beta)^{c_{r,h,w}^{n,s}(\textbf{m},\textbf{m}')}}\Bigg),\label{pnsalpha}
\end{align}}
where the last step follows by using Lemma \ref{lemma4.1}. 
Similarly, on substituting $s=k$ in \eqref{lppnst} and using Eq. (18) of Pote and Kataria (2025), we get
{\scriptsize \begin{align}
\Tilde{p}_{n,k}^{\theta,\alpha}(z)
&=\frac{\Psi_{\theta,\alpha}(z)}{z}\sum_{\substack{m_{j},r\geq0\\\sum_{j=1}^{l}jm_{j} -r=n}}\Big(\prod_{i=1}^{l}\frac{c_{i}^{m_{i}}}{m_{i}!}\Big)\frac{\Lambda^{\sum_{j=1}^{l}m_{j}}(k\mu )^{rk}}{(rk)!} \Big(\int_{0}^{\infty}y^{\sum_{j=1}^{l}m_{j}+rk}e^{-(\Lambda+k\mu+\Psi_{\theta,\alpha}(z))y}\mathrm{d}y\nonumber\\
&\ \ +k\mu\int_{0}^{\infty}\int_{0}^{y}p_{0}(v)(y-v)^{\sum_{j=1}^{l}m_{j}+rk} e^{-(\Lambda+k\mu)(y-v)}e^{-y\Psi_{\theta,\alpha}(z)}\,\mathrm{d}v\, \mathrm{d}y\Big)-\frac{\Psi_{\theta,\alpha}(z)}{z}    
\sum_{\substack{m_{j},r\geq0\\\sum_{j=1}^{l}jm_{j} -r=n}}\frac{c_{1}^{m_{1}+1}}{(m_{1}+1)!}\nonumber\\
&\ \ \cdot\Big(\prod_{i=2}^{l}\frac{c_{i}^{m_{i}}}{m_{i}!}\Big)\frac{\Lambda^{\sum_{j=1}^{l}m_{j}+1}(k\mu )^{k(r+1)}}{(k(r+1)-1)!}\int_{0}^{\infty}\int_{0}^{y}p_{0}(v)(y-v)^{\sum_{j=1}^{l}m_{j}+k(r+1)}e^{-(\Lambda+k\mu)(y-v)}e^{-y\Psi_{\theta,\alpha}(z)}\,\mathrm{d}v\, \mathrm{d}y\,\,\nonumber\\
&=\frac{(z+\theta)^\alpha-\theta^\alpha}{z}\Bigg(\sum_{\substack{m_{j},r\geq0\\\sum_{j=1}^{l}jm_{j} -r=n}}\frac{A_{r}^{n,k}(\textbf{m})}{((z+\theta)^\alpha-\beta)^{a_r^{n,k}(\textbf{m})}}+\sum_{\substack{m_{j},r\geq0\\\sum_{j=1}^{l}jm_{j} -r=n}}\sum_{h=1}^{\infty}\sum_{w=0}^{\infty}\sum_{\substack{\sum_{j=1}^{l}m'_{j}=w\\m'_{j}\geq0}}  \nonumber\\
&\ \ \cdot \frac{B_{r,h,w}^{n,k}(\textbf{m},\textbf{m}')}{((z+\theta)^\alpha-\beta)^{b_{r,h,w}^{n,k}(\textbf{m},\textbf{m}')}}-\sum_{\substack{m_{j},r\geq0\\\sum_{j=1}^{l}jm_{j} -r=n}}\sum_{h=1}^{\infty}\sum_{w=0}^{\infty}\sum_{\substack{\sum_{j=1}^{l}m'_{j}=w\\m'_{j}\geq0}}\frac{C_{r,h,w}^{n,k}(\textbf{m},\textbf{m}')}{((z+\theta)^\alpha-\beta)^{c_{r,h,w}^{n,k}(\textbf{m},\textbf{m}')}}\Bigg).\label{pnkalpha}
\end{align}}
Finally, on taking the inverse Laplace transform on both side of \eqref{pnsalpha} and \eqref{pnkalpha}, we get the required result.
\end{proof}
\begin{remark}
On substituting $l=1$ in Theorem \ref{thm4.3}, we get explicit expression for the state probabilities of Erlang queue time-changed by an independent inverse tempered stable subordinator as given in \eqref{pnstemper}.
\end{remark}

\subsection{Queue length process}
The queue length process $\{\mathcal{L}^{\theta,\alpha}(t)\}_{t\ge0}$ corresponding to the tempered Erlang queue with multiple arrivals is defined as $\mathcal{L}^{\theta,\alpha}(t)\coloneqq g_k(\mathcal{Q}^{\theta,\alpha}(t))$, $t\ge0$. Here, $g_{k}:S\cup \{(0,0)\}\rightarrow\mathbb{N}_{0}$ is a bijective map defined as follows (see Ascione {\it{et al.}} (2020), p. 3252): 
\begin{equation}\label{gk}
g_{k}(n,s)=\left\{
\begin{array}{ll}
k(n-1)+s,\, (n,s)\in S,\\
0,\, (n,s)=(0,0).
\end{array}
\right.  
\end{equation}
Also, its inverse map is given by $(a_{k}(m),b_{k}(m))$, where 
\begin{equation*}
	b_{k}(m)=\left\{
	\begin{array}{ll}
		\min\{s>0: s \equiv m \pmod{k} \},\, m>0,\\
		0,\, m=0,
	\end{array}
	\right.  
\end{equation*}
and
\begin{equation*}
a_{k}(m)=\left\{
	\begin{array}{ll}
		\frac{m-b_{k}(m)}{k}+1,\, m>0,\\
		0,\, m=0.
	\end{array}
	\right.  
\end{equation*}
\begin{remark}
As $g_k(\cdot)$ is a bijective map, the \textit{state-phase} process $\{\mathcal{Q}^{\theta,\alpha}(t)\}_{t\ge0}$ can be represented by its queue length process in terms of phases.
\end{remark}
Let $q_{n}^{\theta,\alpha}(t)=\mathrm{Pr}(\mathcal{L}^{\theta,\alpha}(t)=n|\mathcal{L}^{\theta,\alpha}(0)=0)$, $n\ge0$ denote the state probabilities of $\{\mathcal{L}^{\theta,\alpha}(t)\}_{t\ge0}$ such that $q_{-n}^{\theta,\alpha}(t)=0$ for all $n\geq1$. Then, from \eqref{gk}, we have
\begin{equation}\label{Questprb}
q_{g_k(n,s)}^{\theta,\alpha}(t)=p_{n,s}^{\theta,\alpha}(t)
\end{equation}
and $q_n^{\theta,\alpha}(t)=p_{a_k(n),b_k(n)}^{\theta,\alpha}(t)$. By using \eqref{Questprb} in Theorem \ref{cptfde}, the state probabilities $q_{n}^{\theta,\alpha}(t)$, $n\ge0$ solve the following fractional differential equations:
\begin{align}\label{cptlength}
	\left.
	\begin{aligned}
		\partial_t^{\theta,\alpha}q_{0}^{\theta,\alpha}(t)
		&=-\Lambda q_{0}^{\theta,\alpha}(t)+k\mu q_{1}^{\theta,\alpha}(t), \\
		\partial_t^{\theta,\alpha}q_{n}^{\theta,\alpha}(t)
		&=-(\Lambda + k \mu)q_{n}^{\theta,\alpha}(t)+k\mu q_{n+1}^{\theta,\alpha}(t)+\Lambda\sum_{m=1}^{n}c'_{m}q_{n-m}^{\theta,\alpha}(t), \, n \geq 1,  
	\end{aligned}
	\right\}
\end{align}
where 
\begin{equation*}
	c'_{m}=\begin{cases}
		\lambda_{i}/\Lambda,\,m=ik,i=1,2,\dots,l,\\
		0,\, \text{otherwise},
	\end{cases}
\end{equation*}
with $q_{0}^{\theta,\alpha}(0)=1$ and $q_{n}^{\theta,\alpha}(0)=0,\,n\geq1$.

\begin{theorem}
The mean queue length of tempered Erlang queue with multiple arrivals, that is, $\mathcal{M}^{\theta,\alpha}(t)=\mathds{E}(\mathcal{L}^{\theta,\alpha}(t)|\mathcal{L}^{\theta,\alpha}(0)=0)$ solves the following fractional Cauchy problem:
\begin{equation*}
\partial_t^{\theta,\alpha}\mathcal{M}^{\theta,\alpha}(t)=k\mu p_{0}^{\theta,\alpha}(t)-k\mu+k\Lambda\sum_{i=1}^{l}ic_{i},\
\mathcal{M}^{\theta,\alpha}(0)=0.
\end{equation*}
\end{theorem}
\begin{proof}
It follows that
\begin{equation}\label{dtmt}
\mathcal{M}^{\theta,\alpha}(t)=\sum_{n=0}^{\infty}nq_{n}^{\theta,\alpha}(t), \ \mathcal{M}^{\theta,\alpha}(0)=0.
\end{equation}
On taking the Caputo tempered fractional derivative on both sides of \eqref{dtmt} and by using \eqref{cptlength}, the proof follows similar lines to that of Theorem 9 of Pote and Kataria (2025).
\end{proof}

\begin{theorem}
The mean queue length of tempered Erlang queue with multiple arrivals is given by
{\footnotesize\begin{align*}
\mathcal{M}^{\theta,\alpha}(t)&=k\Big(\Lambda\sum_{i=1}^{l}ic_{i}-\mu\Big)e^{-\theta t}\sum_{m=0}^{\infty}\theta^m t^{m+\alpha}E_{\alpha,\alpha+m+1}(\theta^\alpha t^\alpha)\nonumber\\
&\ \ +k\mu \sum_{h=1}^{\infty}\sum_{n=0}^{\infty}\sum_{\substack{\sum_{j=1}^{l}m_{j}=n\\m_{j}\geq0}}C_{h,n}^{0}(\textbf{m})e^{-\theta t} \sum_{m=0}^{\infty}\theta^m t^{\beta_{h,n}^{\mathcal{M}}(\textbf{m})+m-1}E_{\alpha,\beta_{h,n}^{\mathcal{M}}(\textbf{m})+m}^{\gamma_{h,n}^{0}(\textbf{m})}(\beta t^{\alpha}),\,t\geq0,
\end{align*}}
where  $\beta_{h,n}^{\mathcal{M}}(\textbf{m})=\alpha\gamma_{h,n}^{0}(\textbf{m})+1$.
\end{theorem}
\begin{proof}
Let $\mathcal{M}(t)$ be the mean queue length of Erlang queue with multiple arrivals $\{\mathcal{Q}(t)\}_{t\ge0}$. Then, by using \eqref{pnsalphat} in \eqref{dtmt}, we have
\begin{equation}\label{m1talpha}
\mathcal{M}^{\theta,\alpha}(t)=\int_{0}^{\infty}\mathcal{M}(y)\mathrm{Pr}(Y_{\theta,\alpha}(t)\in\mathrm{d}y),
\end{equation}
which follows on using \eqref{Questprb}. 
On using Eq. (32) of Pote and Kataria (2025) in 
\eqref{m1talpha}, we get
{\small\begin{equation}\label{m1ttt}
\mathcal{M}^{\theta,\alpha}(t)=k\Big(\Lambda\sum_{i=1}^{l}ic_{i}-\mu\Big)\int_{0}^{\infty}y\mathrm{Pr}(Y_{\theta,\alpha}(t)\in\mathrm{d}y)+k\mu\int_{0}^{\infty}\int_{0}^{y}p_{0}(u)\mathrm{d}u\mathrm{Pr}(Y_{\theta,\alpha}(t)\in\mathrm{d}y).
\end{equation}}
By taking the Laplace transform on both sides of \eqref{m1ttt} and using \eqref{lapdeninvtemp}, we obtain
{\small\begin{align}
\Tilde{\mathcal{M}}^{\theta,\alpha}(z)&=k\Big(\Lambda\sum_{i=1}^{l}ic_{i}-\mu\Big)\frac{\Psi_{\theta,\alpha}(z)}{z}\int_{0}^{\infty}ye^{-y\Psi_{\theta,\alpha}(z)}\mathrm{d}y+k\mu \frac{\Psi_{\theta,\alpha}(z)}{z}\int_{0}^{\infty}\int_{0}^{y}p_{0}(u)e^{-y\Psi_{\theta,\alpha}(z)}\mathrm{d}u\mathrm{d}y\nonumber\\
&=k\Big(\Lambda\sum_{i=1}^{l}ic_{i}-\mu\Big)\frac{1}{z\Psi_{\theta,\alpha}(z)}+k\mu \frac{\Psi_{\theta,\alpha}(z)}{z} \int_{0}^{\infty}p_{0}(u)e^{-u\Psi_{\theta,\alpha}(z)}\mathrm{d}u\nonumber\\
&=k\Big(\Lambda\sum_{i=1}^{l}ic_{i}-\mu\Big)\frac{1}{z\Psi_{\theta,\alpha}(z)}+k\mu z^{-1}\sum_{h=1}^{\infty}\sum_{n=0}^{\infty}\sum_{\substack{\sum_{j=1}^{l}m_{j}=n\\m_{j}\geq0}}\frac{C_{h,n}^{0}(\textbf{m})}{((z+\theta)^{\alpha}-\beta)^{\gamma_{h,n}^{0}(\textbf{m})}},\label{meanalphat}
\end{align}}
where the last step follows on using \eqref{p0alpha4}.
Finally, on taking the inverse Laplace transform on both sides of \eqref{meanalphat}, and using \eqref{inlshi1} and \eqref{inlshi2}, we get the required result.
\end{proof}

\subsection{Distribution of busy period}
The busy period in a tempered Erlang queue with multiple arrivals is defined as the duration starting with the arrival of a customer to an initially empty system and ending at the first instance when the system becomes empty again.

Let $\{{\mathcal{L}}_{*}(t)\}_{t\geq0}$ be the process whose state probabilities 
$\mathcal{P}_{n}(t)=\mathrm{Pr}({\mathcal{L}}_{*}(t)=n|{\mathcal{L}}_{*}(0)=a)$, $a\in \{k, 2k,\dots,lk\}$ solve the system of differential equations given in Eq. (36) of Pote and Kataria (2025). Observe that the process $\{{\mathcal{L}}_{*}(t)\}_{t\geq0}$ behaves similar to $\{\mathcal{L}(t)\}_{t\geq0}$ except that it starts from $a$ instead of 0, that is, it represents the arrival of first group of customers in the system, and 0 is its absorbing state. Also, let
\begin{equation}\label{lentemp}
\mathscr{L}_*^{\theta,\alpha}(t)\coloneqq\mathscr{L}_*(Y_{\theta,\alpha}(t))
\end{equation}  
be a time-changed process whose state probabilities are denoted by $\mathscr{P}_n^{\theta,\alpha}(t)=\mathrm{Pr}({\mathcal{L}}_{*}^{\theta,\alpha}(t)=n|{\mathcal{L}}_{*}^{\theta,\alpha}(0)=a)$. 

In the following result, we obtain the distribution of busy period $\mathcal{B}^{\theta,\alpha}$ for $\{\mathcal{Q}^{\theta,\alpha}(t)\}_{t\ge0}$.    
\begin{theorem}
Let  $F_{\mathcal{B}^{\theta,\alpha}}(t)=\mathrm{Pr}\{\mathcal{B}^{\theta,\alpha}\le t\}$. Then, 
{\small \begin{align*}
F_{\mathcal{B}^{\theta,\alpha}}(t)&=a\sum_{n=0}^{\infty}\Lambda^{n}\sum_{\substack{\sum_{j=1}^{l}m_{j}=n\\m_{j}\in\mathbb{N}_{0}}}\Big(\prod_{i=1}^{l}\frac{c_{i}^{m_{i}}}{m_{i}!}\Big)\frac{(k\mu)^{a+k\sum_{j=1}^{l}jm_{j}}(a+n+k\sum_{j=1}^{l}jm_{j}-1)!}{(a+k\sum_{j=1}^{l}jm_{j})!}\\
&\hspace{2.8cm}\cdot \sum_{m=0}^{\infty}\theta^mt^{\alpha(a+n+k\sum_{j=1}^{l}jm_j)+m}E_{\alpha,\alpha(a+n+k\sum_{j=1}^{l}jm_j)+m+1}^{a+n+k\sum_{j=1}^{l}jm_j} (\beta t^\alpha), \ t\geq0.\\
\end{align*}}
\end{theorem}
\begin{proof}
It follows from the construction of process $\{{\mathscr{L}}_{*}^{\theta,\alpha}(t)\}_{t\geq0}$ that $F_{B^{\theta,\alpha}}(t)=\mathscr{P}_0^{\theta,\alpha}(t)$. By using \eqref{lentemp}, we have
\begin{align}
F_{\mathcal{B}^{\theta,\alpha}}(t)&=\int_{0}^{\infty}F_\mathcal{B}(y)\mathrm{Pr}\{Y_{\theta,\alpha}(t)\in \mathrm{d}y\}\nonumber\\
&=a\sum_{n=0}^{\infty}\Lambda^{n}\sum_{\substack{\sum_{j=1}^{l}m_{j}=n\\m_{j}\in\mathbb{N}_{0}}}\Big(\prod_{i=1}^{l}\frac{c_{i}^{m_{i}}}{m_{i}!}\Big)\frac{(k\mu)^{a+k\sum_{j=1}^{l}jm_{j}}}{(a+k\sum_{j=1}^{l}jm_{j})!}\nonumber\\
&\hspace{3cm}\cdot\int_{0}^{\infty}\int_{0}^{y}z^{a+n+k\sum_{j=1}^{l}jm_{j}-1} e^{-(\Lambda+k\mu)z}\mathrm{d}z\mathrm{Pr}\{Y_{\theta,\alpha}(t)\in \mathrm{d}y\},\label{busy1}
\end{align}
which follows on using Eq. (35) of Pote and Kataria (2025). On taking the Laplace transform on both sides of \eqref{busy1}, we get
{\footnotesize\begin{align}
\tilde{F}_{\mathcal{B}^{\theta,\alpha}}(t)&=a\sum_{n=0}^{\infty}\Lambda^{n}\sum_{\substack{\sum_{j=1}^{l}m_{j}=n\\m_{j}\in\mathbb{N}_{0}}}\Big(\prod_{i=1}^{l}\frac{c_{i}^{m_{i}}}{m_{i}!}\Big)\frac{\Psi_{\theta,\alpha}(v)(k\mu)^{a+k\sum_{j=1}^{l}jm_{j}}}{v(a+k\sum_{j=1}^{l}jm_{j})!}\nonumber\\
&\hspace{6cm}\cdot\int_{0}^{\infty}\int_{0}^{y}z^{a+n+k\sum_{j=1}^{l}jm_{j}-1} e^{-(\Lambda+k\mu)z}e^{-y\Psi_{\theta,\alpha}(v)}\mathrm{d}z\mathrm{d}y\nonumber\\
&=a\sum_{n=0}^{\infty}\Lambda^{n}\sum_{\substack{\sum_{j=1}^{l}m_{j}=n\\m_{j}\in\mathbb{N}_{0}}}\Big(\prod_{i=1}^{l}\frac{c_{i}^{m_{i}}}{m_{i}!}\Big)\frac{\Psi_{\theta,\alpha}(v)(k\mu)^{a+k\sum_{j=1}^{l}jm_{j}}}{v(a+k\sum_{j=1}^{l}jm_{j})!}\nonumber\\
&\hspace{6cm}\cdot\int_{0}^{\infty}\int_{z}^{\infty}z^{a+n+k\sum_{j=1}^{l}jm_{j}-1} e^{-(\Lambda+k\mu)z}e^{-y\Psi_{\theta,\alpha}(v)}\mathrm{d}y\mathrm{d}z\nonumber\\
&=a\sum_{n=0}^{\infty}\Lambda^{n}\sum_{\substack{\sum_{j=1}^{l}m_{j}=n\\m_{j}\in\mathbb{N}_{0}}}\Big(\prod_{i=1}^{l}\frac{c_{i}^{m_{i}}}{m_{i}!}\Big)\frac{(k\mu)^{a+k\sum_{j=1}^{l}jm_{j}}}{v(a+k\sum_{j=1}^{l}jm_{j})!}\nonumber\\
&\hspace{6cm}\cdot\int_{0}^{\infty}z^{a+n+k\sum_{j=1}^{l}jm_{j}-1} e^{-(\Lambda+k\mu+\Psi_{\theta,\alpha}(v))z}\mathrm{d}z\nonumber\\
&=a\sum_{n=0}^{\infty}\Lambda^{n}\sum_{\substack{\sum_{j=1}^{l}m_{j}=n\\m_{j}\in\mathbb{N}_{0}}}\Big(\prod_{i=1}^{l}\frac{c_{i}^{m_{i}}}{m_{i}!}\Big)\frac{v^{-1}(a+n+k\sum_{j=1}^{l}jm_{j}-1)!(k\mu)^{a+k\sum_{j=1}^{l}jm_{j}}}{(a+k\sum_{j=1}^{l}jm_{j})!(\Lambda+k\mu+\Psi_{\theta,\alpha}(v))^{a+n+k\sum_{j=1}^{l}jm_{j}}}.\label{busy3}
\end{align}}
Finally, on taking the inverse Laplace transform on both sides of \eqref{busy3} and using \eqref{inlshi2}, we get the required result.
\end{proof}
\section{Tempered Erlang queue}\label{se4}
Here, we study the Erlang queue time-changed by the first hitting time of tempered stable subordinator. We call it the tempered Erlang queue. 

First, we recall the definition of an Erlang queue. It is a queue in which the arrivals are modeled according to a Poisson process with rate $\lambda>0$ and the service system is composed of $k$ phases where each phase has exponential service time with parameter $k\mu$. 
Let $N(t)$ be the number of customers in the system at time $t$ and let $S(t)$ be the phase of customer being served at time $t$. Then, the state-phase process $Q(t)=(N(t),S(t))$, $t\ge0$ is an Erlang queue. 

Let $\{Y_{\theta,\alpha}(t)\}_{t\ge0}$ be an inverse tempered stable subordinator which is independent of $\{Q(t)\}_{t\ge0}$. We define the tempered Erlang queue $\{Q^{\theta,\alpha}(t)\}_{t\ge0}$, $\theta>0$, $0<\alpha<1$ as follows:
\begin{equation}\label{temperdef}
Q^{\theta,\alpha}(t)\coloneqq Q(Y_{\theta,\alpha}(t)),\ t\ge0,
\end{equation}
and denote its state probabilities by
\begin{align*}
\mathcal{P}_{0}^{\theta,\alpha}(t)&=\mathrm{Pr}(Q^{\theta,\alpha}(t)=(0,0)|Q^{\theta,\alpha}(0)=(0,0)),\\
\mathcal{P}_{n,s}^{\theta,\alpha}(t)&=\mathrm{Pr}(Q^{\theta,\alpha}(t)=(n,s)|Q^{\theta,\alpha}(0)=(0,0)),\, (n,s)\in S,
\end{align*}
where $S=\{(n,s) \in \mathbb{N} \times \mathbb{N}: n \geq 1, \, 1 \leq s \leq k \}$.	

On substituting $l=1$ in Theorem \ref{cptfde}, the state probabilities of tempered Erlang queue solve the following system of fractional differential equations: 
{\small\begin{align}\label{deertemp}
\left.\begin{aligned}
\partial_t^{\theta,\alpha}\mathcal{P}_{0}^{\theta,\alpha}(t)
&=-\lambda \mathcal{P}_{0}^{\theta,\alpha}(t)+k\mu \mathcal{P}_{1,1}^{\theta,\alpha}(t),  \\
\partial_t^{\theta,\alpha}\mathcal{P}_{1,s}^{\theta,\alpha}(t)
&=-(\lambda + k \mu)\mathcal{P}_{1,s}^{\theta,\alpha}(t)+k\mu \mathcal{P}_{1,s+1}^{\theta,\alpha}(t), \, 1 \leq s \leq k-1, \\
\partial_t^{\theta,\alpha}\mathcal{P}_{1,k}^{\theta,\alpha}(t)
&=-(\lambda + k \mu)\mathcal{P}_{1,k}^{\theta,\alpha}(t)+k\mu \mathcal{P}_{2,1}^{\theta,\alpha}(t)+ \lambda \mathcal{P}_{0}^{\theta,\alpha}(t),  \\
	\partial_t^{\theta,\alpha}\mathcal{P}_{n,s}^{\theta,\alpha}(t)
	&=-(\lambda + k \mu)\mathcal{P}_{n,s}^{\theta,\alpha}(t)+k\mu \mathcal{P}_{n,s+1}^{\theta,\alpha}(t)+\lambda \mathcal{P}_{n-1,s}^{\theta,\alpha}(t), \   1\leq s \leq k-1,  \, n \geq 2,  \\
	\partial_t^{\theta,\alpha}\mathcal{P}_{n,k}^{\theta,\alpha}(t)
	&=-(\lambda + k \mu)\mathcal{P}_{n,k}^{\theta,\alpha}(t)+k\mu \mathcal{P}_{n+1,1}^{\theta,\alpha}(t)+\lambda \mathcal{P}_{n-1,k}^{\theta,\alpha}(t),\, n\geq2,
\end{aligned}\right\}
\end{align}}
with $\mathcal{P}_{0}^{\theta,\alpha}(0)=1$ and
$\mathcal{P}_{n,s}^{\theta,\alpha}(0)=0$, $1\leq s\leq k$, $n\geq 1$.

In the subsequent results, we will be using the following notations:
   $\eta=\theta^\alpha-\lambda-k\mu$,
$\psi_{\theta,\alpha}(z)=(z+\theta)^\alpha-\theta^\alpha$, $d_{r}^{n,s}=n+r+k(r+1)-s+1$, $\zeta_{r}^{n,s}=\alpha(d_{r}^{n,s}-1)+1$, 
\begin{align*}
		\delta_{h,n}^0&=h+n(k+1),\\
		Z^0_{h,n}&=\frac{h\lambda^n}{n!}\frac{(k\mu)^{\delta_{h,n}^0}(\delta_{h,n}^0-1)!}{(\delta_{h,n}^0-n)!},\\ \psi^0_{h,n}&=\alpha(\delta_{h,n}^0-1)+1,\\
		D_{r}^{n,s}&=\frac{\lambda^{n+r}(k\mu )^{k(r+1)-s}}{(n+r)!(k(r+1)-s)!}(d_{r}^{n,s}-1)!,\\
		F_{r,h,w}^{n,s}&=k\mu D_{r}^{n,s}Z_{h,w}^{0},\\
		Z_{r,h,w}^{n,s}&=\begin{cases}
			k\mu D_{r}^{n,s+1}Z_{h,w}^{0},\,s\ne k,\\
			k\mu D_{r}^{n+1,1}Z_{h,w}^{0},\, s=k,
		\end{cases}	\\
		z_{r,h,w}^{n,s}&=\begin{cases}
			d_{r}^{n,s+1}+\delta_{h,w}^{0},\,s\ne k,\\
			d_{r}^{n+1,1}+\delta_{h,w}^{0},\, s=k,
		\end{cases}
\end{align*}
$f_{r,h,w}^{n,s}=d_{r}^{n,s}+\delta_{h,w}^{0}$, $\eta_{r,h,w}^{n,s}=\alpha(f_{r,h,w}^{n,s}-1)+1$
and 
$\alpha_{r,h,w}^{n,s}=\alpha(z_{r,h,w}^{n,s}-1)+1$, where $n$, $r$, $w$ are non-negative integers and $h$ is a positive integer. 

On substituting $l=1$ in Theorem \ref{theorem4.2}, the zero state probability $\mathcal{P}_{0}^{\theta,\alpha}(t)$ of the tempered Erlang queue is given by
{\small\begin{equation}\label{ertempzero}
\mathcal{P}_{0}^{\theta,\alpha}(t)=\sum_{h=1}^{\infty}\sum_{n=0}^{\infty}\sum_{m=0}^{\infty}Z_{h,n}^{0}e^{-\theta t}\theta^mt^{\Psi_{h,n}^{0}+m-1}\Big(E_{\alpha,\Psi_{h,n}^{0}+m}^{\delta_{h,n}^{0}}(\eta t^{\alpha})-(\theta t)^\alpha E_{\alpha,\Psi_{h,n}^{0}+\alpha+m}^{\delta_{h,n}^{0}}(\eta t^{\alpha})\Big).
\end{equation}}

Also, for $n\geq1$ and $1\leq s\leq k$, and on taking $l=1$ in Theorem \ref{thm4.3}, the state probabilities $\mathcal{P}_{n,s}^{\theta,\alpha}(t)$ of tempered Erlang queue are given by  		
{\footnotesize \begin{align}\label{pnstemper}
\mathcal{P}_{n,s}^{\theta,\alpha}(t)&=\sum_{r=0}^{\infty}D_{r}^{n,s}\sum_{a=0}^{\infty}e^{-\theta t}\theta^\alpha t^{\zeta_{r}^{n,s}+a-1}\Big(E_{\alpha,\zeta_{r}^{n,s}+a}^{d_{r}^{n,s}}(\eta t^{\alpha})-(\theta t)^\alpha E_{\alpha,\zeta_{r}^{n,s}+\alpha +a}^{d_{r}^{n,s}}(\eta t^{\alpha})\Big)\nonumber\\
&\ \ +\sum_{r=0}^{\infty}\sum_{h=1}^{\infty}\sum_{w=0}^{\infty}F_{r,h,w}^{n,s}\sum_{b=0}^{\infty}e^{-\theta t}\theta^bt^{\eta_{r,h,w}^{n,s}+b-1}\Big(E_{\alpha,\eta_{r,h,w}^{n,s}+b}^{f_{r,h,w}^{n,s}}(\eta t^{\alpha})-(\theta t)^\alpha E_{\alpha,\eta_{r,h,w}^{n,s}+\alpha +b}^{f_{r,h,w}^{n,s}}(\eta t^{\alpha})\Big)\nonumber\\
&\ \  -\sum_{r=0}^{\infty}\sum_{h=1}^{\infty}\sum_{w=0}^{\infty}Z_{r,h,w}^{n,s}\sum_{c=0}^{\infty}e^{-\theta t}\theta^c t^{\alpha_{r,h,w}^{n,s}+c-1}\Big(E_{\alpha,\alpha_{r,h,w}^{n,s}+c}^{z_{r,h,w}^{n,s}}(\eta t^{\alpha})-(\theta t)^\alpha E_{\alpha,\alpha_{r,h,w}^{n,s}+\alpha +c}^{z_{r,h,w}^{n,s}}(\eta t^{\alpha}) \Big).
\end{align}}

\subsection{Mean queue length}
Here, we discuss about the mean queue length of tempered Erlang queue. It is defined as follows:
\begin{equation}\label{meantemper}
M^{\theta,\alpha}(t)\coloneqq\mathbb{E}(L^{\theta,\alpha}(t)|L^{\theta,\alpha}(0)=0).
\end{equation}
Here, $\{L^{\theta,\alpha}(t)\}_{t\ge0}$ is the queue length proces of tempered Erlang queue which is defined as $L^{\theta,\alpha}(t)=g_k(Q^{\theta,\alpha}(t)), \ t\ge0$, where $g_k$ is a bijective map as given in \eqref{gk}.

Let   $r_{n}^{\theta,\alpha}(t)=\mathrm{Pr}(L^{\theta,\alpha}(t)=n|L^{\theta,\alpha}(0)=0)$, $n\ge0$ denote the state probabilities of $\{L^{\theta,\alpha}(t)\}_{t\ge0}$ such that $r_{-n}^{\theta,\alpha}(t)=0$ for all $n\geq1$. Then, from \eqref{gk}, we have
\begin{equation}\label{Questprbb}
r_{g_k(n,s)}^{\theta,\alpha}(t)=\mathcal{P}_{n,s}^{\theta,\alpha}(t)
\end{equation}
and  $r_n^{\theta,\alpha}(t)=\mathcal{P}_{a_k(n),b_k(n)}^{\theta,\alpha}(t)$. From \eqref{deertemp} and \eqref{Questprbb}, it follows that the state probabilities $r_{n}^{\theta,\alpha}(t)$, $n\ge0$ are solution of 
\begin{align}\label{sysdeerlen}
\left.
\begin{aligned}
\partial_t^{\theta,\alpha}r_{0}^{\theta,\alpha}(t)
&=-\lambda r_{0}^{\theta,\alpha}(t)+k\mu r_{1}^{\theta,\alpha}(t), \\
\partial_t^{\theta,\alpha}p_{n}^{\theta,\alpha}(t)
&=-(\lambda + k \mu)r_{n}^{\theta,\alpha}(t)+k\mu r_{n+1}^{\theta,\alpha}(t), \, 1\leq n \leq k-1, \\
\partial_t^{\theta,\alpha}r_{n}^{\theta,\alpha}(t)
&=-(\lambda + k \mu)r_{n}^{\theta,\alpha}(t)+k\mu r_{n+1}^{\theta,\alpha}(t)+\lambda r_{n-m}^{\theta,\alpha}(t), \, n \geq k,		 
\end{aligned}
\right\}
\end{align}
with $r_{0}^{\theta,\alpha}(0)=1$ and $r_{n}^{\theta,\alpha}(0)=0,\,n\geq1$. 

From \eqref{meantemper} and \eqref{sysdeerlen}, the following result holds:
\begin{theorem}
The mean queue length of tempered Erlang queue solves the following fractional Cauchy problem:
\begin{equation*}
\partial_t^{\theta,\alpha}{M}^{\theta,\alpha}(t)=k\mu r_{0}^{\theta,\alpha}(t)+k(\lambda-\mu),
\end{equation*}
with initial condition $\mathcal{M}^{\theta,\alpha}(0)=0$.
\end{theorem}

\begin{theorem}
The mean queue length of tempered Erlang queue is given by
\begin{align*}
M^{\theta,\alpha}(t)&=k(\lambda-\mu)e^{-\theta t}\sum_{m=0}^{\infty}\theta^m t^{m+\alpha}E_{\alpha,\alpha+m+1}(\theta^\alpha t^\alpha)\nonumber\\
&\ \ +k\mu \sum_{h=1}^{\infty}\sum_{n=0}^{\infty}Z_{h,n}^{0}e^{-\theta t} \sum_{m=0}^{\infty}\theta^m t^{\Psi_{h,n}^{M}+m-1}E_{\alpha,\Psi_{h,n}^{M}+m}^{\delta_{h,n}^{0}}(\eta t^{\alpha}),\,t\geq0,
\end{align*}
where  $\Psi_{h,n}^{M}=\alpha\delta_{h,n}^{0}+1$.
\end{theorem}
\begin{proof}
Let $M(t)$ be the mean queue length of Erlang queue $\{Q(t)\}_{t\ge0}$. Then, by using \eqref{meantemper}, we have
\begin{equation}\label{m1talphaa}
M^{\theta,\alpha}(t)=\int_{0}^{\infty}M(y)\mathrm{Pr}(Y_{\theta,\alpha}(t)\in\mathrm{d}y).
\end{equation}
The remaining part of the proof follows similar steps  to that of Theorem 6.2 of Ascione \textit{et al.} (2020), with Eq. (21) of Luchak (1958) substituted for $M(y)$ in \eqref{m1talphaa}.
\end{proof}
\subsection{Inter-arrival and inter-phase times of tempered Erlang queue}
Here, we discuss the tempered Erlang queue in terms of its inter-arrival, inter-phase and sojourn times. 

Let $h_{n}^{\theta,\alpha}(t)=\mathrm{Pr}(N^{\theta,\alpha}(t)=n|Q^{\theta,\alpha}(0)=(0,0)),\,n\geq1$ and $
g_{s}^{\theta,\alpha}(t)=\mathrm{Pr}(S^{\theta,\alpha}(t)=s|Q^{\theta,\alpha}(0)=(0,0)),\,1\leq s\leq k$. 
Then, we have
\begin{equation}\label{hntemer}
h_{n}^{\theta,\alpha}(t)=\sum_{s=1}^{k}\mathcal{P}_{n,s}^{\theta,\alpha}(t) \text{ and }
g_{s}^{\theta,\alpha}(t)=\sum_{n=1}^{\infty}\mathcal{P}_{n,s}^{\theta,\alpha}(t),
\end{equation}
which follows from \eqref{temperdef}.
For $1\leq s\leq k$, by using \eqref{deertemp} and \eqref{hntemer}, we get
\begin{align*}
\partial_t^{\theta,\alpha}\mathcal{P}_{0}^{\theta,\alpha}(t)&=-\lambda \mathcal{P}_{0}^{\theta,\alpha}(t)+k\mu \mathcal{P}_{1,1}^{\theta,\alpha}(t),\\
\partial_t^{\theta,\alpha}h_{1}^{\theta,\alpha}(t)&=-\lambda h_{1}^{\theta,\alpha}(t)+k\mu(\mathcal{P}_{2,1}^{\theta,\alpha}(t)-\mathcal{P}_{1,1}^{\theta,\alpha}(t))+\lambda \mathcal{P}_{0}^{\theta,\alpha}(t),\\
\partial_t^{\theta,\alpha}h_{n}^{\theta,\alpha}(t)&=-\lambda h_{n}^{\alpha_{1},\alpha_{2}}(t)+k\mu(\mathcal{P}_{n+1,1}^{\theta,\alpha}(t)-\mathcal{P}_{n,1}^{\theta,\alpha}(t))+\lambda h_{n-1}^{\theta,\alpha}(t),\,n\geq 2,
\end{align*} 
with $\mathcal{P}_{0}^{\theta,\alpha}(0)=1$ and $h_{n}^{\theta,\alpha}(0)=0$, $n\geq1$.

Let $\{N_{*}(t)\}_{t\geq0}$ be a process whose state probabilities $p^*_{n}(t)=\mathrm{Pr}(N_{*}(t)=n)$, $n\geq0$ solve the following system of differential equations:
\begin{align}\label{qmDE12}
\left.
\begin{aligned}
\frac{\mathrm{d}}{\mathrm{d}t}p^*_{m}(t)&=-\lambda p^*_{m}(t),\\
\frac{\mathrm{d}}{\mathrm{d}t} p^*_{m+1}(t)&=\lambda p^*_{m}(t),\\
\frac{\mathrm{d}}{\mathrm{d}t}p^*_{n}(t)&=0,\,n\geq0,\,n\neq m, \,m+1,
\end{aligned}
\right\}
\end{align}
with $p^*_{m}(0)=1$ and $p^*_{n}(0)=0$ for $n\geq0$, $n\neq m$. Here, $\{N_*(t)\}_{t\geq0}$ is a pure birth process with rate parameter $\lambda$ as that of the arrival rate of $\{N(t)\}_{t\geq0}$, and it starts at $m$ with $m+1$ as its absorbing state. Note that $p^*_{m}(t)+p^*_{m+1}(t)=1$, $t\geq0$.

Recall that the inter-arrival time is the time elapsed between two consecutive arrivals. In the following result, we obtain the distribution of inter-arrival times of tempered Erlang queue.
\begin{theorem}\label{thminterarrival}
The inter-arrival times $T^{\theta,\alpha}$ of tempered Erlang queue $\{Q^{\theta,\alpha}(t)\}_{t\geq0}$ are independent and identically distributed (iid) with tempered Mittag-Leffler distribution. That is,
{\small\begin{align}\label{intrarr12}
\mathrm{Pr}(T^{\theta,\alpha}>t)&=e^{-\theta t}\sum_{m=0}^{\infty}(\theta t)^m\sum_{r=0}^{\infty}(-\lambda t^\alpha)^rE_{\alpha, r\alpha+m+1}^r(\theta^\alpha t^\alpha),
\end{align}}
where $E_{\alpha,r\alpha+m+1}^{r}(\cdot)$ is the three parameter Mittag-Leffler function as defined in \eqref{mitag}.
\end{theorem}
\begin{proof}
From the definition of tempered Erlang queue, its inter-arrival times are independent. 
Let $T$ denote the arrival time of first new customer in the Erlang queue $\{Q(t)\}_{t\geq0}$ starting from $N_{*}(0)=m$. Then, its distribution function is given by
	\begin{align*}
		F_{T}(t)\coloneq\mathrm{Pr}(T\leq t)&=1-\mathrm{Pr}(T>t)\\
		&=1-\mathrm{Pr}(N_{*}(t)=m)\\
		&=\mathrm{Pr}(N_{*}(t)=m+1)=p^*_{m+1}(t).
	\end{align*} 
Let us consider the time-changed process
	$N_{*}^{\theta,\alpha}(t)=N_{*}(Y_{\theta,\alpha}(t))$, $t\geq0$
	with state probabilities
	\begin{equation*}
	\kappa_m^{\theta,\alpha}(t) = \Pr\big(N_{*}^{\theta,\alpha}(t) = m\big) \quad \text{and} \quad \kappa_{m+1}^{\theta,\alpha}(t) = \Pr\big(N_{*}^{\theta,\alpha}(t) = m+1\big).
	\end{equation*}
	The process $\{N_{*}^{\theta,\alpha}(t)\}_{t \geq 0}$ starts at state $m$, where $m+1$ is absorbing state. Define $T^{\theta,\alpha}$ as the arrival time of the first new customer in the tempered Erlang queue $\{Q^{\theta,\alpha}(t)\}_{t \geq 0}$, given $N_{*}^{\theta,\alpha}(0) = m$. Then, the distribution function of $T^{\theta,\alpha}$ is
	\begin{equation*}
	F_{T^{\theta,\alpha}}(t) = \Pr\big(T^{\theta,\alpha} \leq t\big) = \kappa_{m+1}^{\theta,\alpha}(t).
	\end{equation*}
	
Note that $\{N^{\theta,\alpha}(t)\}_{t\geq0}$ is a semi-Markov process because $\{N(t)\}_{t\geq0}$ is a Markov process. Let us denote its $n$th jump time by $T_{n}$. From Gikhman and Skorohod (1975), it follows that the discrete time  process $N^{\theta,\alpha}(T_{n})$ is a time-homogeneous Markov chain. Therefore, $T^{\theta,\alpha}$ can be considered as the inter-arrival times of tempered Erlang queue $\{Q^{\theta,\alpha}(t)\}_{t\geq0}$.
	
We need to show that the state probabilities of $\{N_{*}^{\theta,\alpha}(t)\}_{t\geq0}$ solve the following differential equations:
\begin{align}
\partial_t^{\theta,\alpha}  \kappa_{m}^{\theta,\alpha}(t)&=-\lambda  \kappa_{m}^{\theta,\alpha}(t),\label{tqm12}\\
\partial_t^{\theta,\alpha}  \kappa_{m+1}^{\theta,\alpha}(t)&=\lambda  \kappa_{m}^{\theta,\alpha}(t)\label{tqmm12},
\end{align}
with $  \kappa_{m}^{\theta,\alpha}(0)=1$ and $  \kappa_{m+1}^{\theta,\alpha}(0)=0$.
	
The Eq. \eqref{tqm12} holds true if and only if the Laplace transform $\tilde{\kappa}_{m}^{\theta,\alpha}(z)$ of $  \kappa_{m}^{\theta,\alpha}(t)$ solves:
\begin{equation}\label{qmDElp12}
\frac{\Psi_{\theta,\alpha}(z)}{z}(z\tilde{  \kappa}_{m}^{\theta,\alpha}(z)-1)=-\lambda\tilde{  \kappa}_{m}^{\theta,\alpha}(z),\,z>0,
\end{equation}
where we have used \eqref{laplacetempcapu}.	
Note that
\begin{equation}\label{temstr}
  \kappa_{m}^{\theta,\alpha}(t)=\mathrm{Pr}(N_{*}(Y_{\theta,\alpha}(t))=m)
=\int_{0}^{\infty}  \kappa_{m}(y)\mathrm{Pr}\{Y_{\theta,\alpha}(t)\in \mathrm{d}y\}.
\end{equation}
On taking the Laplace transform on both sides of \eqref{temstr} and by using \eqref{lapdeninvtemp}, we obtain 
\begin{equation}\label{lpqm12}
\tilde{  \kappa}_{m}^{\theta,\alpha}(z)=\frac{\Psi_{\theta,\alpha}(z)}{z}\int_{0}^{\infty}  \kappa_{m}(y)e^{-y\Psi_{\theta,\alpha}(z)} \mathrm{d}y.
\end{equation}
Now, on substituting \eqref{lpqm12} in \eqref{qmDElp12}, we get
\begin{equation}\label{qmDElp121}
\psi_{\theta,\alpha}(z)\int_{0}^{\infty}  \kappa_{m}(y)e^{-y\Psi_{\theta,\alpha}(z)} \mathrm{d}y-\frac{\psi_{\theta,\alpha}(z)}{z}=-\lambda\int_{0}^{\infty}  \kappa_{m}(y)e^{-y\Psi_{\theta,\alpha}(z)} \mathrm{d}y.
\end{equation}
By using \eqref{qmDE12} in \eqref{qmDElp121} and simplifying, we get \eqref{tqm12}. Similarly, it can be established that \eqref{tqmm12} holds true. Finally, the proof is complete on solving \eqref{tqm12} using $Z$-transform method, the technique used by  Gupta \textit{et al.} (2020) in the proof of Proposition 3.4.
\end{proof}
\begin{theorem}\label{thminterphase12}
The inter-phase times $I^{\theta,\alpha}$ of tempered Erlang queue are iid, and has the following distribution:
{\small\begin{align}\label{intrarr112}
\mathrm{Pr}(I^{\theta,\alpha}>t)&=e^{-\theta t}\sum_{m=0}^{\infty}(\theta t)^m\sum_{r=0}^{\infty}(-k\mu t^\alpha)^rE_{\alpha, r\alpha+m+1}^r(\theta^\alpha t^\alpha),
\end{align}}
where $E_{\alpha, r\alpha+m+1}^r(\cdot)$ is the three parameter Mittag-Leffler function defined in \eqref{mitag}.
\end{theorem}
\begin{proof}
The proof follows similar lines to that of Theorem \ref{thminterarrival}. Thus, it is omitted. 
\end{proof}

\begin{remark}\label{2cor12}
Let $Y_{1},Y_{2},\dots, Y_{k}$ be the inter-phase times of tempered Erlang queue and $X_k=Y_{1}+Y_{2}+\dots +Y_{k}$. As $Y_{i}$'s are iid, the Laplace transform of its pdf is given by 
\begin{equation}\label{servlp12}
	\tilde{f}_{X_k}(z)=\Big(\frac{k\mu}{k\mu+(z+\theta)^{\alpha}-\theta^{\alpha}}\Big)^{k},
\end{equation} 
where we have used the Laplace transform of the pdf of $Y_{1}$ (see Meerschaert \textit{et al.} (2011), Example 5.7).
\end{remark}
Next, we give the distribution of the sojourn time of tempered Erlang queue, that is, the time spent in a particular state before its transition to another state. 
Its proof follows similar lines to that of Theorem \ref{thminterarrival}. Thus, it is omitted.  
\begin{theorem}\label{thmsojourn}
The sojourn times $S^{\theta,\alpha}$ of queue length process $\{L^{\theta,\alpha}(t)\}_{t\geq0}$ for a non-zero state are iid, and has the following distribution function:
{\small\begin{align*}
\mathrm{Pr}(S^{\theta,\alpha}>t)&=e^{-\theta t}\sum_{m=0}^{\infty}(\theta t)^m\sum_{r=0}^{\infty}(-(\lambda+k\mu) t^\alpha)^rE_{\alpha, r\alpha+m+1}^r(\theta^\alpha t^\alpha),
\end{align*}}
where $E_{\alpha,r\alpha+m+1}^r(\cdot)$ is the three parameter Mittag-Leffler function defined in \eqref{mitag}.
\end{theorem}
\begin{remark}
Observe that the sojourn time $S^{\theta,\alpha}$ of queue length process of tempered Erlang queue in state $0$ is the arrival time of first customer in an empty system. Thus, it is distributed according to \eqref{intrarr12}.
\end{remark}
\begin{remark}
Following along the lines of the proof of Corollary 4.6 of Ascione \textit{et al.} (2020), it can be established that the inter-arrival times and inter-phase times of the tempered Erlang queue are not necessarily independent. 
\end{remark}

%

\subsection{Distribution of conditional waiting times} The derivation of waiting times in single-channel queues relies on the Markov property of queues (see Gaver (1954)). The same approach cannot be used in the case of tempered Erlang queue as it lacks the Markov property and belongs to the class of semi-Markov processes. As obtained in Ascione et al. (2020), some information about the waiting times of tempered Erlang queue can be derived by applying suitable conditioning techniques.

For fix $t_0>0$, let us consider a random variable $X$ which has the following distribution function:
\begin{equation}\label{surdef}
F_X(t)=\mathrm{Pr}(X\le t_0+t|X> t_0)=1-\frac{\phi^{\theta,\alpha}(t_0+t)}{\phi^{\theta,\alpha}(t_0)},
\end{equation}
with parameters $t_0$ and $k\mu$, where $\phi^{\theta,\alpha}(t)=\mathrm{Pr}\{I^{\theta,\alpha}>t\}$ as given in \eqref{intrarr112}.

From \eqref{surdef}, the pdf of $X$ can be written as
\begin{equation}\label{surpdf}
f_X(t)=-\frac{1}{\phi^{\theta,\alpha}(t_0)}\frac{\mathrm{d}}{\mathrm{d}t}\phi^{\theta,\alpha}(t_0+t).
\end{equation}
On taking the Laplace transform of \eqref{surpdf}, we get
\begin{align}\label{lpsurden}
\tilde{f}_X(z)&=-\frac{1}{\phi^{\theta,\alpha}(t_0)}(z\tilde{\phi}^{\theta,\alpha}(t_0+t)-\phi^{\theta,\alpha}(0))\nonumber\\
&=1-\frac{1}{\phi^{\theta,\alpha}(t_0)}\int_{0}^{\infty}e^{-zt}\phi^{\theta,\alpha}(t_0+t)\, \mathrm{d}t\nonumber\\
&=1-\frac{ze^{-\theta t_0}}{\phi^{\theta,\alpha}(t_0)}\sum_{m=0}^{\infty}\theta^m\sum_{r=0}^{\infty}(-k\mu)^r\nonumber\\
&\ \ \cdot\sum_{x=0}^{\infty}\frac{\Gamma(r+x)\theta^{\alpha x}}{x!\Gamma(r)\Gamma(\alpha(r+x)+m+1)} \int_{0}^{\infty}e^{-t(z+\theta)}(t_0+t)^{\alpha(r+x)+m}\,\mathrm{d}t\nonumber\\
&=1-\frac{ze^{-\theta t_0}}{\phi^{\theta,\alpha}(t_0)}\sum_{m=0}^{\infty}\theta^m\sum_{r=0}^{\infty}(-k\mu)^r\nonumber\\
&\ \ \cdot \sum_{x=0}^{\infty}\frac{\Gamma(r+x)\theta^{\alpha x}}{x!\Gamma(r)\Gamma(\alpha(r+x)+m+1)}\frac{e^{(z+\theta)t_0}\Gamma(\alpha(r+x)+m+1,(z+\theta)t_0)}{(z+\theta)^{\alpha(r+x)+m+1}},
\end{align}
where the penultimate step follows on using \eqref{intrarr112}. Here, $\Gamma(x,y)$ is the analytic extension of the upper incomplete gamma function on $\mathbb{C}^2$ (see Gradshteyn \textit{et al.} (2007)).

The waiting time of a customer entering in the system at time $t>0$ is the amount of time that he has to spend in the service. We denote it by $W^{\theta,\alpha}(t)$. Also, we define a random set $\mathscr{S}_t^{\theta,\alpha}$ which represents the collection of time instants before $t$ at which a service phase has been completed, that is, 
\begin{equation*}
	\mathscr{S}_{t}^{\theta,\alpha}({\omega})=\{0\leq s \leq t: \mathcal{L}^{\theta,\alpha}(s^{-})(\omega)=\mathcal{L}^{\theta,\alpha}(s)(\omega)+1\}.
\end{equation*}
 Let $\mathcal{M}^{\theta,\alpha}_{t}=|\mathscr{S}_{t}^{\theta,\alpha}|$ be a random variable which counts the number of phases completed before time $t$. Also, let $\mathcal{Y}^{\theta,\alpha}_{t}$ be the last instant before $t$ at which a phase has been completed, that is,
\begin{equation*}
	\mathcal{Y}^{\theta,\alpha}_{t}=\begin{cases}
		\sup \mathscr{S}_{t}^{\theta,\alpha},\, \mathcal{M}^{\theta,\alpha}_{t}>0,\\
		t,\, \mathcal{M}^{\theta,\alpha}_{t}=0.
	\end{cases}
\end{equation*}

For a fixed $t>0$, the random time  $\mathcal{Y}^{\theta,\alpha}_{t}$ is a Markov time that coincides with the jump times of semi-Markov process $\{S^{\theta,\alpha}(t)\}_{t\geq0}$ except in the case when $\mathcal{Y}^{\theta,\alpha}_{t}=t$. The semi-regenerative set consists of the discontinuity points of $\{S^{\theta,\alpha}(t)\}_{t\geq0}$ as it is a semi-Markov process. If $\mathcal{Y}^{\theta,\alpha}_{t}=t_{0}$ for some $t_{0}<t$, the past and future of $\{S^{\theta,\alpha}(t)\}_{t\geq0}$ are independent (see Cinlar (1974)). Thus, for waiting time, we condition on $\{\mathcal{L}^{\theta,\alpha}(t)=n\}$ and $\{\mathcal{Y}^{\theta,\alpha}_{t}=t_{0}\}$. 

Let
\begin{equation*}
w^{\theta,\alpha}(u;t,t_{0},n)\,\mathrm{d}u=\mathrm{Pr}(W_{t}^{\theta,\alpha}\in \mathrm{d}u|\mathcal{Y}^{\theta,\alpha}_{t}=t_{0},\mathcal{L}^{\theta,\alpha}(t)=n).
\end{equation*} 	

For $1\leq j\leq n-1$, let $T_{j}$ denote the duration of $j$th phase whose service is not yet started and $T_{n}$ corresponds to the remaining service time of the phase in progress at time $t$. From Theorem \ref{thminterphase12}, it follows that $T_{j}$, $1\leq j\leq n-1$ are iid. Thus, $W=T_1+T_2+\dots+T_n$ has the distribution given by the sum of $n-1$ iid random variables with distribution given in \eqref{intrarr112}. For the distribution of $T_{n}$, it is known that the last phase was completed at $t_{0}$ as we have conditioned on $\mathcal{Y}^{\theta,\alpha}_{t}=t_{0}$. So, the phase being served at time $t$ started its service at time $t_{0}$. Thus, $T_{n}$ has a distribution given in \eqref{surdef} with parameters $t-t_{0}$ and $k\mu$.

Conditioning on the event $\{\mathcal{L}^{\theta,\alpha}(t)=n\}$, a customer must wait for the completion of $n$ independent phases. Among these, one of the $n$ phases began at $t_{0}$ and remains unfinished at $t$. That is, 
\begin{equation*}
\mathbb{E}(W^{\theta,\alpha}_{t}|\mathcal{Y}^{\theta,\alpha}_{t}=t_{0},\mathcal{L}^{\theta,\alpha}(t)=n)=\sum_{j=1}^{n}T_{j}.
\end{equation*}
Now, by using the independence of $T_{j}$'s, $j=1,2,\dots,n$, we have
\begin{equation}\label{wtden12}
	w^{\theta,\alpha}(u;t,t_{0},n)=f_{{W}}*f_{T_{n}}(u).
\end{equation}
On taking the Laplace transform of \eqref{wtden12}, and using \eqref{servlp12} and \eqref{lpsurden}, we get
{\small\begin{align*}
	\tilde{w}(z;t,t_{0},n)
	&=\Big(\frac{k\mu}{k\mu+(z+\theta)^\alpha-\theta^\alpha}\Big)^{n-1}\Bigg(1-\frac{ze^{-\theta (t-t_0)}}{\phi^{\theta,\alpha}(t-t_0)}\sum_{m=0}^{\infty}\theta^m\sum_{r=0}^{\infty}(-k\mu)^r\\ &\ \ \cdot\sum_{x=0}^{\infty}\frac{\Gamma(r+x)\theta^{\alpha x}}{x!\Gamma(r)\Gamma(\alpha(r+x)+m+1)}\frac{e^{(z+\theta)(t-t_0)}\Gamma(\alpha(r+x)+m+1,(z+\theta)(t-t_0))}{(z+\theta)^{\alpha(r+x)+m+1}}\Bigg).
\end{align*} }

\section*{Acknowledgement}
The	first author thanks Government of India for the grant of Prime Minister's Research Fellowship, ID 1003066.


\begin{thebibliography}{00}

\bibitem{Alrawashdeh2017}
Alrawashdeh, M. S., Kelly, J. F., Meerschaert, M. M. and Scheffler, H.-P. (2017). Applications of inverse tempered stable subordinators. \textit{Comput. Math. Appl.} \textbf{73}(6), 892-905.


		
\bibitem{Ascione2018}
Ascione, G., Leonenko, N. and Pirozzi, E. (2018). Fractional queues with catastrophes and their transient behaviour. \textit{Math.} \textbf{6}(9), 159.

\bibitem{Ascione2020}
Ascione, G., Leonenko, N. and Pirozzi, E. (2020). Fractional Erlang queues. \textit{Stochastic Process. Appl.} \textbf{130}(6), 3249-3276.


\bibitem{Cahoy2015}
Cahoy, D. O., Polito, F. and Phoha, V. (2015). Transient behavior of fractional queues and related processes. \textit{Methodol. Comput. Appl. Probab.} \textbf{17}(3), 739-759.

\bibitem{Chen2020}
Chen, A., Wu, X. and Zhang, J. (2020). Markovian bulk-arrival and bulk-service queues with
general state-dependent control. \textit{Queueing Syst.} \textbf{95}(3-4), 331-378.

\bibitem{Cinlar1974}
Cinlar, E. (1974). Markov additive process and semi-regeneration. {\it Discussion Paper No. 118}, North
western University.
 

\bibitem{Dhillon2024}
Dhillon, M. and Kataria, K. K. (2024). On martingale characterizations of generalized counting process and its time-changed variants. \textit{J. Math. Anal. Appl.} \textbf{504}(2), 1-7.

\bibitem{DiCrescenzo2003}
Di Crescenzo, A., Giorno, V., Kumar, B. K. and Nobile, A. G. (2003). On the $M/M/1$ queue with catastrophes and its continuous approximation. \textit{Queueing Syst.} \textbf{43}(4), 329-347.

\bibitem{Di Crescenzo2016}	
Di Crescenzo, A., Martinucci, B. and Meoli, A. (2016). A fractional counting process and its connection with the Poisson process. \textit{ALEA Lat. Am. J. Probab. Math. Stat.} {\bf13}(1), 291-307.


\bibitem{Fomundam2007}
Fomundam, S. F. and Herrmann, J. W. (2007). A survey of queuing theory applications in healthcare. https://drum.
lib.umd.edu/handle/1903/7222

				
\bibitem{Gaver1954}
Gaver, D. P. (1954). The influence of servicing times in queuing processes. \textit{J. Operations Res. Soc. Amer.} \textbf{2}(2), 139-149.

\bibitem{Giambene2014}
Giambene, G. (2014). Queuing Theory and Telecommunications, Vol. 585. Springer, Berlin.

\bibitem{Gihman1975}
Gikhman, I. I. and Skorokhod, A. V. (1975). The Theory of Stochastic Processes. {II}. 218.

\bibitem{Giorno2018}
Giorno, V., Nobile, A. G. and Pirozzi, E. (2018). A state-dependent queueing system with asymptotic logarithmic distribution. \textit{J. Math. Anal. Appl.} \textbf{458}(2), 949-966.

\bibitem{Gradshteyn2007}
Gradshteyn., I. S., Ryzhik., I. M. Jeffrey., A. and Zwillinger., D. (2007). Table of Integrals, Series, and Products (Seventh Edition). Academic Press, Boston.

\bibitem{Griffiths2006}
Griffiths, J. D., Leonenko, G. M. and Williams, J. E. (2006). The transient solution to $M/E_{k}/1$ queue. \textit{Oper. Res. Lett.} \textbf{34}, 349-354.


\bibitem{Gupta2020}
Gupta, N., Kumar, A. and Leonenko, N. (2020). Tempered fractional Poisson processes and fractional equations with Z-transform. \textit{Stoch. Anal. Appl.} \textbf{38}(5), 939–957.

\bibitem{Katria2022}
Kataria, K. K. and Khandakar, M. (2022). Generalized fractional counting process. \textit{J. Theoret. Probab.} {\bf35}(4), 2784-2805.

\bibitem{Kilbas2004}
Kilbas, A. A., Saigo, M. and Saxena, R. K. (2004). Generalized Mittag-Leffler Function and Generalized Fractional Calculus Operators. \textit{Integral Transforms. Spec. Funct.} {\bf15}(1), 31–49.

\bibitem{Kilbas2006}
Kilbas, A. A., Srivastava, H. M. and Trujillo, J. J. (2006). Theory and Applications of Fractional Differential Equations. Elsevier Science B.V., Amsterdam.

\bibitem{Luchak1956}
Luchak, G. (1956). The solution of the single-channel queuing equations characterized by a time-dependent Poisson-distributed arrival rate and a general class of holding times. \textit{Operations Res.} \textbf{4}(6), 711-732.

\bibitem{Luchak1958}
Luchak, G. (1958). The continuous time solution of the equations of the single channel queue with a general  class of service-time distributions by the method of generating functions.
\textit{J. Roy. Statist. Soc. Ser. B.} \textbf{20}(1), 176-181.

	\bibitem{Meerschaert2011}
Meerschaert, M. M., Nane, E. and Vellaisamy, P. (2011). The fractional Poisson process and the inverse stable subordinator. {\it Electron. J. Probab.} {\bf 16}(59), 1600-1620.

\bibitem{Meerschaert2013}
Meerschaert, M. M., Nane, E. and Vellaisamy, P. (2013). Transient anomalous subdiffusions on bounded domains. \textit{Proc. Amer. Math. Soc.} \textbf{141}, 699-710.

\bibitem{Pote2025}
Pote, R. B. and Kataria, K. K. (2025).  On Erlang queue with multiple arrivals and its time-changed variant. \textit{Methodol. Comput. Appl. Probab.} \textbf{27}, 69.

\bibitem{Spath2006}
Spath, D. and F{\"a}hnrich, K. P. (2006). Advances in Services Innovations. Springer, Berlin, Heidelberg.




	
\end{thebibliography}
\end{document}